\documentclass[oneside, 12pt]{amsart}
\usepackage{amsmath}
\usepackage{amssymb}
\usepackage{color}
\usepackage[usenames,dvipsnames,x11names,svgnames]{xcolor}
\usepackage[colorlinks=true,linkcolor=NavyBlue,citecolor=DarkGreen]{hyperref}
\usepackage{amsthm}
\usepackage{amscd}
\usepackage{geometry}
\usepackage{mathrsfs}
\usepackage{dsfont}

\newtheorem{thm}{Theorem}

\newtheorem{prop}{Proposition}
\newtheorem{lem}[prop]{Lemma}

\theoremstyle{definition}

\newtheorem*{rem}{Remark}
\newtheorem*{rems}{Remarks}

\renewcommand{\pmod}[1]{\,(\textup{mod}\,#1)}

\newcommand{\mb}{\mathbb}
\newcommand{\mc}{\mathcal}

\newcommand{\ol}{\overline}
\newcommand{\ul}{\underline}
\newcommand{\leqs}{\leqslant }
\newcommand{\geqs}{\geqslant }

\newcommand{\be}{\begin{equation*}}
\newcommand{\ee}{\end{equation*} }

\newcommand{\ben}{\begin{equation}}
\newcommand{\een}{\end{equation} }

\newcommand{\bs}{\begin{split}}
\newcommand{\es}{\end{split}}

\newcommand{\bmu}{\begin{multline*}}
\newcommand{\emu}{\end{multline*}}

\newcommand{\bmun}{\begin{multline}}
\newcommand{\emun}{\end{multline}}

\begin{document}
\title[Discrete negative moments]{Lower bounds for discrete negative moments of the Riemann zeta function}

\author{Winston Heap}
\address{Max Planck Institute for Mathematics, Vivatsgasse 7, 53111 Bonn.}
\email{winstonheap@gmail.com}

\author{Junxian Li}
\address{Max Planck Institute for Mathematics, Vivatsgasse 7, 53111 Bonn.}
\email{jli135@mpim-bonn.mpg.de}

\author{Jing Zhao}
\address{Max Planck Institute for Mathematics, Vivatsgasse 7, 53111 Bonn.}
\email{jingzh95@gmail.com}

\maketitle

\begin{abstract}
We prove lower bounds for the discrete negative $2k$th moment of the derivative of the Riemann zeta function for all fractional $k\geqslant 0$. The bounds are in line with a conjecture of Gonek and Hejhal. Along the way, we  prove a general formula for the discrete twisted second moment of the Riemann zeta function. This agrees with a conjecture of Conrey and Snaith.
\end{abstract}


\section{Introduction}

We are interested in the discrete negative moments 
\begin{equation}
J_{-k}(T)=\sum_{0\leqs \gamma\leqs T} \frac{1}{|\zeta^\prime(\rho)|^{2k}}
\end{equation}
when $k$ is fractional. These can give information about small values of the derivative $\zeta^\prime(\rho)$, and in the special case $k=1/2$ are related to partial sums of the M\"obius function (e.g. see Theorem 14.27 of \cite{T}). A natural assumption when considering these moments is that all the non-trivial zeros of the zeta function are simple. We therefore assume this throughout the paper unless otherwise mentioned. 

Gonek \cite{G} and Hejhal \cite{H} independently conjectured that 
\[
J_{-k}(T)\asymp T(\log T)^{(k-1)^2}
\]
for all $k\geqs 0$. However, the range of $k$ in which this conjecture holds seems to be in doubt since Gonek (unpublished) has suggested that there exist infinitely many zeros $\rho$ for which $\zeta^\prime(\rho)^{-1}\gg \gamma^{1/3-\epsilon}$, in which case the conjecture would fail for $k>3/2$. Hughes, Keating and O'Connell \cite{HKO} used random matrix theory to predict a precise constant in this conjecture. Interestingly, their formulas on the random matrix theory side undergo a phase change at the point $k=3/2$ which gives alternative evidence to suggest that the conjecture fails for $k>3/2$. 

Aside from these conjectures, little is known about $J_{-k}(T)$. Assuming the Riemann hypothesis (RH) Gonek \cite{G} showed that 
\[J_{-1}(T)\gg T\]
and that by applying H\"older's inequality
\[J_{-k}(T)\gg T(\log T)^{1-3k}\]
for all $k>0$. Milinovich and Ng \cite{MN} later refined Gonek's bound by showing that 
\[J_{-1}(T)\geqs (1+o(1)) \frac{3}{2\pi^3} T.\] 
The value of the constant here is half the conjectured value \cite{G, HKO}.
In the special case $k=1/2$, Heath-Brown \cite[pg.\! 386]{T}, shows
\[
J_{-1/2}(T)\gg T
\] 
via the connection with $\sum_{n\leqs x}\mu(n)$ on assuming RH. Our aim in this paper is to improve these lower bounds.


\begin{thm}\label{main thm}Assume RH and that all zeros are simple. Then
\begin{equation}
J_{-k}(T)\gg T(\log T)^{(k-1)^2}
\end{equation}
for all fractional $k\geqs 0$.
\end{thm}

In particular, our theorem gives the improved bound
\[
J_{-1/2}(T)\gg T(\log T)^{1/4}.
\]
We recently learned that this lower bound has been obtained independently by Milinovich, Ng and Soundararajan using similar methods. In their method however there is no need to assume that all zeros are simple.

As we detail below, our proof utilises the method of Rudnick and Soundararajan \cite{RS} and shares some similarities with the paper of Chandee and Li \cite{CL} who gave lower bounds for fractional moments of Dirichlet $L$-functions in the $q$-aspect. A general side effect of the method employed by Chandee and Li, which our results also share, is that the implicit constants depend on the height of the rational number $k$. We mention that we give a couple of simplifications to their argument which reduces the length and in some cases allows for an asymptotic evaluation of the multidimensional Mellin integrals which commonly feature in this area (see \cite{KMVK} for instance).

Let us outline the method of proof. Here and throughout, let
\[
k=a/b
\]
with $a,b\in\mathbb{N}$. 
To prove Theorem \ref{main thm} we apply H\"older's inequality in the form 
\begin{equation}\label{holder}
  \sum_{0\leqs \gamma\leqs T} |P(\rho)|^{2a}
  \leqs 
  \Big(\sum_{0\leqs \gamma\leqs T}|\zeta^\prime(\rho)|^2|P(\rho)|^{2(a+b)}\Big)^{\frac{a}{a+b}}
  \Big(\sum_{0\leqs \gamma\leqs T}\frac{1}{|\zeta^\prime(\rho)|^{2k}}\Big)^{\frac{b}{a+b}}.
\end{equation}
where 
\begin{equation}
\label{poly}
P(s)=\sum_{n\leqs x}\frac{\tau_{-1/b}(n)}{n^s}\psi(n),
\end{equation}
and $x=T^{\theta/(a+b)}$ with $\theta<1/2$. Here, $\tau_{\alpha}(n)$ denotes the Dirichlet series coefficients of $\zeta(s)^\alpha$, $\alpha\in\mb{C}$, and $\psi(n)$ is a smoothing weight which will be properly defined later (see formula \eqref{psi choice} below).
 We then have the following two propositions.

\begin{prop}
\label{S_1 prop} Let $P(s)$ be given by \eqref{poly} and let 
\[
S_1:= \sum_{0\leqs \gamma\leqs T} P(\rho)^{a}P(1-\rho)^{a}.
\]
Then for fixed $a,b\in\mathbb{N}$, 
\[
S_1\sim c(a,b) T(\log  T)^{\frac{a^2}{b^2}+1}
\]
for some positive constant $c(a,b)$ as $T\to\infty$.
\end{prop}

\begin{prop}
\label{S_2 prop}
Let $P(s)$ be given by \eqref{poly} and let 
\[
S_2:= \sum_{0\leqs \gamma\leqs T} \zeta^\prime(\rho)\zeta^\prime(1-\rho)P(\rho)^{a+b}P(1-\rho)^{a+b}.
\]
Then for fixed $a,b\in\mathbb{N}$, 
\[
S_2\ll_{a,b} T(\log  T)^{\frac{a^2}{b^2}+3}
\]
as $T\to\infty$.
\end{prop}

Note these propositions are unconditional. 
Assuming RH the sums $S_1$ and $S_2$ become those in H\"older's inequality \eqref{holder} giving
\[J_{-k}(T)\geqs \frac{S_1^{k+1}}{S_2^k}\gg  T \frac{(\log T)^{(k+1)(k^2+1)}}{(\log T)^{k(k^2+3)}}=T(\log T)^{(k-1)^2}\]
and Theorem \ref{main thm} follows.

In  proving Proposition \ref{S_1 prop} we apply a result of Ng \cite{Ng} which deals with sums of the $S_1$-type when $P(s)$ is a fairly general Dirichlet polynomial. We have not been able to find such a general formula for sums of the $S_2$-type in the literature, however these sums have been dealt with in specific cases \cite{HBB,CGG, CGGSimple, Ng1}. Following their methods, we derive the following general result.

\begin{thm} 
\label{twisted moment thm}
Let $\alpha,\beta \ll 1/\log T$ be sufficiently small shifts. Let $Q(s)=\sum_{n\leqs y}a(n)n^{-s}$ with $y=T^\theta$ and $\theta<1/2$ and denote $\ol{Q}(s)=\sum_{n\leqs y}\ol{a(n)}n^{-s}$. Suppose that $a(mn)\ll |a(m)a(n)|$ and $a(n)\ll \tau_r(n)(\log n)^C$.   Denote
\begin{equation}\label{f}
f_{\alpha,\gamma}(n)=\sum_{n_1n_2n_3=n}\mu(n_1)n_2^{-\alpha}n_3^{-\gamma},
\end{equation}
and let
\ben
\label{Z}
    Z_{\alpha,\beta,\gamma,h,k}
  =\frac{1}{h^\beta}
    \frac{\zeta(1+\alpha+\beta)\zeta(1+\beta+\gamma)}{\zeta(1+\beta)}
     \prod_{p^{k_p} || k}\frac{\sum_{m\geqs 0}f_{\alpha,\gamma}(p^{m+k_p})p^{-m(1+\beta)}}{\sum_{m\geqs 0}f_{\alpha,\gamma}(p^{m})p^{-m(1+\beta)}}.
\een
Furthermore, let 
\begin{multline}
\label{mc I}
\mc{I}(\alpha,\beta,T)=\sum_{g\leqs y}\sum_{\substack{h,k\leqs y/g\\(h,k)=1}}\frac{a(gh)\ol{a(gk)}}{ghk}
               \frac{d}{d\gamma}
              \frac{1}{2\pi} \int_0^T\bigg[ 
               Z_{\alpha,\beta,\gamma,h,k}
     \\  + \bigg(\frac{t}{2\pi}\bigg)^{-\alpha-\beta}Z_{-\beta,-\alpha,\gamma,h,k}
      + \bigg(\frac{t}{2\pi}\bigg)^{-\beta-\gamma}Z_{\alpha,-\gamma,-\beta,h,k}
     \bigg]dt
     \bigg|_{\gamma=0}
\end{multline}
and
\begin{multline}
\label{mc J}
          \mc{J}(\alpha,\beta,T)
          =
            \sum_{g\leqs y}\sum_{\substack{h,k\leqs y/g\\(h,k)=1}}\frac{a(gh)\ol{a(gk)}}{ghk}
              \frac{1}{2\pi}\int_0^T\log\bigg(\frac{t}{2\pi}\bigg)
              \\\times\bigg[
              \frac{\zeta(1+\alpha+\beta)}{h^{\beta}k^{\alpha}}
      + \bigg(\frac{t}{2\pi}\bigg)^{-\alpha-\beta}
             \frac{\zeta(1-\alpha-\beta)}{h^{-\alpha}k^{-\beta}}
             \bigg]dt.
\end{multline}
Then
\begin{multline}
\label{twisted moment eq}
          \sum_{0\leqs \gamma\leqs T}\zeta(\rho+\alpha)\zeta(1-\rho+\beta)Q(\rho)\ol{Q}(1-\rho)\\
  =  \mc{I}(\alpha,\beta,T)+\ol{\mc{I}(\ol{\beta},\ol{\alpha},T)}+\mc{J}(\alpha,\beta,T)+(T(\log T)^{-A})
\end{multline}
where $A$ is an arbitrary positive constant.  
\end{thm}

\begin{rems}
${}$

$\bullet$ We note that this result is unconditional. Initially, results of this type required the assumption of GRH \cite{CGG} to analyse the error term. By an application of the large sieve, this condition was later weakened to the Generalised Lindel\"of hypothesis by Conrey, Ghosh and Gonek \cite{CGGSimple} and was finally made unconditional by Bui and Heath-Brown \cite{HBB} utilising Heath-Brown's identity. We follow the latter method when analysing our error terms. 

$\bullet$ If one assumes GRH then one can allow for general complex coefficients satisfying $a(n)\ll n^\epsilon$. In this case the the error term is replaced by $O(y^{3/2+\epsilon}T^{1/2+\epsilon})$, and so in this setting one has an asymptotic formula  provided $\theta<1/3$.

$\bullet$ Our main term takes the form as predicted by the recipe method/ratios conjecture \cite{CFKRS,CS}. Indeed, the $Z$ terms can be realised as a diagonal sum (see Lemma \ref{diag lem} below).

$\bullet$ In terms of applications, the form stated in this theorem is not ideal. The polar behaviour of the integrand and the arithmetic sums can be a little troublesome. We have refrained from stating a more `applicable' version of the theorem (as in Theorem 1.2 of \cite{BBLR}) however all the necessary ingredients for applications can be found in Section \ref{S2Proof} (see formulas \eqref{contour int rep}, \eqref{contour int rep 2} and Lemma \ref{diag lem} below).  

\end{rems}

Before moving on to the proofs we give a brief heuristic justification for the choice of Dirichlet polynomial $P(s)$ in \eqref{poly} since we could not find this elswhere in the literature. In order for H\"older's inequality to be sharp we need the summands to be approximately equal. Unfortunately, $1/\zeta(s)$ has no representation in terms of a Dirichlet series and so there is no obvious choice for a polynomial approximation. However, we expect that our mean values will not change too much if we shift away from the half-line slightly. We also expect that in this region $\zeta^\prime(s)\approx (\log T) \zeta(s)$, at least on average. Applying these two principles the right hand side of \eqref{holder} becomes
 \[
 \Big(\sum_{0\leqs \gamma\leqs T}|\zeta(\rho+\delta)|^2|P(\rho+\delta)|^{2(a+b)}\Big)^{\frac{a}{a+b}}
  \Big(\sum_{0\leqs \gamma\leqs T}\frac{1}{|\zeta(\rho+\delta)|^{2k}}\Big)^{\frac{b}{a+b}}
\]
and notice that the powers of $\log T$ cancel by homogeniety. We may now set our summands equal:
\[
\zeta(s)^2P(s)^{2(a+b)}\approx \frac{1}{\zeta(s)^{2a/b}},
\]
and find that we should take $P(s)\approx \zeta(s)^{-\frac{1}{b}}$. 

The paper is organized as follows. In Section \ref{S1Proof} we prove Proposition \ref{S_1 prop}. In section \ref{S2Proof} we prove Proposition \ref{S_2 prop} assuming Theorem \ref{twisted moment thm}. The remainder of the paper is then devoted to proving Theorem \ref{twisted moment thm}. In Section \ref{Theorem2Proof}, we compute the main term and bound the error term using two propositions: Propositions \ref{E1} and \ref{E2}. Then in Section \ref{SmallmoduliProof} we prove Proposition \ref{E1}, and in Section \ref{LargemoduliPropProof} we prove Proposition \ref{E2}.


\section{Proof of Proposition \ref{S_1 prop}}\label{S1Proof}
We first choose the weight function $\psi(n)$ in \eqref{poly}. Let $B$ be a positive integer and let
\begin{equation}
\label{psi choice}
\psi(n)=\mathds{1}_{n\leqs x}\bigg(\frac{\log(x/n)}{\log x}\bigg)^B.
\end{equation}
We will need to take $B$ sufficiently large in terms of $k$ at several points throughout the paper, so for the moment we keep it general. To be clear, in the end we will choose 
\[
B=14+12k
\] 
but for the purposes of this section we only require $B\geqs 1$. 
Note that by the Mellin inversion formula (or simply by a residue computation) we have 
\begin{equation}
\label{mellin}
\psi(n)=\frac{B!}{2\pi i(\log x)^B  }\int_{(c)}\bigg(\frac{x}{n}\bigg)^s\frac{ds}{s^{B+1}}
\end{equation}
for $c>0$ where, here and throughout, $\int_{(c)}=\int_{c-i\infty}^{c+i\infty}$.

We first write
\[
R(s)=P(s)^a=\sum_{n\leqs x^a}\frac{r(n,x)}{n^s}
\]
with
\[
r(n,x)=r_{a,b}(n,x)=\sum_{\substack{n_1\cdots n_a=n \\ n_j\leqs x}}\tau_{-1/b}(n_1)\psi(n_1)\cdots \tau_{-1/b}(n_a)\psi(n_a).
\]
Then 
\[
S_1=\sum_{0\leqs \gamma\leqs T}R(\rho)R(1-\rho).
\]
The mean value $S_1$ can be computed in a familiar manner; either by writing it as a contour integral or by, what amounts to the same thing, applying Gonek's uniform version Landau's formula.
These details have been carried out by Ng \cite{Ng} for a fairly general class of Dirichlet polynomial. By applying Proposition 4, (i) of \cite{Ng} we find that 
\begin{equation}
\begin{split}
\label{Ngs}
        S_1
    = & N(T)\sum_{n\leqs x^a}\frac{r(n,x)^2}{n}
     -\frac{T}{\pi}\sum_{\ell m=n\leqs x^a}\frac{\Lambda(\ell)r(m,x)r(n,x)}{n}+o(T)
\end{split}
\end{equation}
where $N(T)=\frac{T}{2\pi}\log(T/2\pi e)+O(\log T)$ is the number of zeros of $\zeta(s)$ in the strip $0<\sigma<1$, $0\leqs t\leqs T$. 

Denote
\[
S_{11}=\sum_{n\leqs x^a}\frac{r(n,x)^2}{n}
\]
and 
\[
S_{12}=\sum_{\ell m=n\leqs x^a}\frac{\Lambda(\ell)r(m,x)r(n,x)}{n}.
\]
In this section we will show that
\[S_{11}\sim c_1(a,b)(
\log T)^{a^2/b^2}  \]
and that 
\[S_{12}\sim c_2(a,b)(
\log T)^{a^2/b^2+1}  \]
for some explicit constants $c_1(a,b)$, $c_2(a,b)$. We then show that $\tfrac{1}{2}c_1(a,b)-c_2(a,b)>0$ and hence 
$S_1\sim c_3(a,b)T(\log T)^{a^2/b^2+1}$ for some positive constant $c_3(a,b)$. The result will then follow. 
We compute $S_{11}$ first.

\subsection{Computing $S_{11}$}
Unfolding the sum we have 
 
\[
 S_{11}=\sum_{\substack{n_1\cdots n_a=\\n_{a+1}\cdots n_{2a}\\ n_j\leqs x}}\frac{\tau_{-1/b}(n_1)\cdots \tau_{-1/b}(n_{2a})}{(n_1\cdots n_{2a})^{1/2}} \psi(n_1)\cdots\psi(n_{2a})
\]
By applying the Mellin inversion formula \eqref{mellin} in each $n_j$ and interchanging the order of summation and integration we get 
\[
S_{11}=\frac{B!^{2a}}{(2\pi i)^{2a}( \log x)^{2aB}}\int_{(c)^{2a}}
\sum_{\substack{n_1\cdots n_a=\\n_{a+1}\cdots n_{2a}}}
\frac{\tau_{-1/b}(n_1)\cdots \tau_{-1/b}(n_{2a})}{n_1^{1/2+s_1}\cdots n_{2a}^{1/2+s_{2a}}}
 \prod_{\ell=1}^{2a}x^{s_\ell}\frac{ds_\ell}{s_\ell^{B+1}},  
\]
where we have taken $c={ 1 }/{\log x}$. Here, we use the notation
\[
\int_{(c)^{2a}}=\underbrace{\int_{(c)}\cdots\int_{(c)}}_{2a}.
\]
After a short calculation with Euler products we find that the Dirichlet series in the integrand is given by 
\[
 \mc{A}(\ul{s}) \prod_{i,j=1}^a\zeta(1+s_i+s_{a+j})^{1/b^2} 
\]
where 
\begin{equation}
\label{A(s)}
  \mc{A}(\ul{s}) = \prod_p\prod_{i,j=1}^a\left( 1 - \frac{ 1 }{ p^{1+s_i+s_{a+j}} } \right)^{1/b^2}\sum_{ \substack{ m_1+\cdots+m_a=\\ m_{a+1}+\cdots+m_{2a}\\m_j\geqslant 0 } } \frac{ \tau_{-1/b}(p^{m_1})\cdots \tau_{-1/b}(p^{m_{2a}}) }{ p^{ m_1(\frac{1}{2}+s_1 )+\cdots+m_{2a}( \frac{1}{2}+s_{2a} ) } }.
\end{equation}
Note that $A(\ul{s})$ is holomorphic in the region $\sigma_j>-1/4$, $j=1,\ldots,2a$, since it is absolutely convergent there. 

We will reproduce the following argument several times throughout the paper so we take this opportunity to briefly describe the steps and give some justification. Note that the integrand has fractional powers of $\zeta(s)$. This coupled with the fact that we have a multidimensional integral means that shifting contours would be very messy. However, note that the integrand is largest when we simultaneously have  $\Im(s_j)\approx 0$. We can therefore localise our integral around these points, expand the integrand in Taylor/Laurent series, and then extract the main term via the substitution $s_j\mapsto s_j/\log x$. The remaining integral then gives a combinatorial constant which we can compute as the weighted volume of a polytope using a trick from \cite{BH} (see also \cite{HL}). This method can essentially be thought of as a multidimensional version of the saddle point method, although in our case it is fairly easy to see where the saddle/main contribution is.

In practice it is simpler if we make the substitutions first, so let $s_j\mapsto s_j/\log x$ for each $j$. Then $ S_{11}$ becomes
\[
 S_{11}=\frac{B!^{2a}}{(2\pi i)^{2a}}\int_{(1)^{2a}}
  \mc{A}(\frac{\ul{s}}{\log x}) 
 \prod_{i,j=1}^a\zeta\left(1+\frac{s_i+s_{a+j}}{\log x}\right)^{1/b^2} 
 \prod_{\ell=1}^{2a}e^{s_\ell}\frac{ds_\ell}{s_\ell^{B+1}},  
\]
Let us localise the integral. For each $j$ we split the integral at the points $t_j=\Im(s_j)=\pm\sqrt{\log x}$; the main contribution will come from the integral over the region $s_j\in [1 - i\sqrt{\log x}, 1 + i\sqrt{\log x}]$. To estimate the tail integrals we use the bound
\[
 \mc{A}(\ul{s}/\log x)\prod_{i,j=1}^a \zeta\left(1+\frac{s_i+s_{a+j}}{\log x}\right)^{1/b^2}\ll (\log x)^{a^2/b^2}
\]
valid for $s_j=1+it_j$ uniformly in $t_j\in\mathbb{R}$. Then,
\begin{align*}
 \int_{1+i\sqrt{\log x}}^{1+i\infty} & \int_{(1)^{2a-1}}
  \mc{A}(\frac{\ul{s}}{\log x}) \prod_{i,j=1}^a\zeta\left(1+\frac{s_i+s_{a+j}}{\log x}\right)^{1/b^2} \prod_{\ell=1}^{2a}e^{s_\ell}\frac{ds_\ell}{s_\ell^{B+1}}
\\
\ll &
(\log x)^{a^2/b^2}\int_{1+i\sqrt{\log x}}^{1+i\infty}\frac{ds}{|s|^{B+1}}\ll (\log x)^{a^2/b^2-1/2}
\end{align*}
by absolute convergence. Naturally, the tail integrals in the lower half plane satisfy the same bound as do those with respect to the other integration variables. Note that the smooth weights $\psi(n)$ have made the task of estimating these tails significantly easier compared to the case of the usual Perron's formula. 

Collecting the errors gives
\begin{multline*}
S_{11}
=
\frac{B!^{2a}}{(2\pi i)^{2a}}\int_{1-i\sqrt{\log x}}^{1+i\sqrt{\log x}} \!\!\!\cdots \!\!\int_{1-i\sqrt{\log x}}^{1+i\sqrt{\log x}}
 \mc{A}(\frac{\ul{s}}{\log x}) \prod_{i,j=1}^a\zeta\left(1+\frac{s_i+s_{a+j}}{\log x}\right)^{1/b^2} \prod_{\ell=1}^{2a}e^{s_\ell}\frac{ds_\ell}{s_\ell^{B+1}}
\\
+O\big( (\log x)^{a^2/b^2-1/2}\big).
\end{multline*}
In this region of integration we have the expansions
\begin{equation}
\label{A taylor}
 \mc{A}(\ul{s}/\log x)= \mc{A}(\ul{0})+O\big(\frac{1}{\log x}\sum_{j} |s_j|\big)= \mc{A}(\ul{0})+O\big(\frac{1}{\sqrt{\log x}}\big)
\end{equation}
and 
\begin{equation}
\label{zeta laurent}
\zeta\left(1+\frac{s_i+s_{a+j}}{\log x}\right)^{1/b^2}=\frac{(\log x)^{1/b^2}}{(s_i+s_{a+j})^{1/b^2}}\Big(1+O\big(\frac{1}{\sqrt{\log x}}\big)\Big).
\end{equation}
Therefore,
\begin{multline*}
 S_{11} =  \frac{ \mc{A}(\ul{0})B!^{2a}\left(\log x\right)^{a^2/b^2}}{(2\pi i )^{2a}} \int_{1-i\sqrt{\log x}}^{1+i\sqrt{\log x}} \!\!\!\cdots \int_{1-i\sqrt{\log x}}^{1+i\sqrt{\log x}} \prod_{i,j=1}^a\frac{1}{(s_i+s_{a+j})^{1/b^2}} 
 \prod_{\ell=1}^{2a}e^{s_\ell}\frac{ds_\ell}{s_\ell^{B+1}}
\\  +O\big( (\log x)^{a^2/b^2-1/2}\big).
\end{multline*}
On extending any given integral to $i\infty$ we acquire a multiplicative error of $O((\log x)^{-B/2})$ which leads to a total contribution of size $O((\log x)^{a^2/b^2-1/2})$. Therefore, we acquire the following asymptotic formula, 
\begin{equation}
\label{S_11 asymp}
S_{11}=  \mc{A}(\ul{0}) \beta(a,b) (\log x)^{a^2/b^2}+O\big( (\log x)^{a^2/b^2-1/2}\big)
\end{equation}
where
\[
\beta(a,b)=
 \frac{B!^{2a} }{(2\pi i )^{2a}} \int_{(1)^{2a}} \prod_{i,j=1}^a\frac{1}{(s_i+s_{a+j})^{1/b^2}} 
 \prod_{\ell=1}^{2a}e^{s_\ell}\frac{ds_\ell}{s_\ell^{B+1}}.
\]
We postpone the computation of these constants to subsection \ref{constants sec}.

\subsection{Computing $S_{12}$}
Recall
\[
S_{12} = \sum_{\ell m=n\leqs x^a}\frac{\Lambda(\ell)r(m,x)r(n,x)}{n}.
\]
In order to have multiplicative coefficients we write  
$
\Lambda(n) = \sum_{n_1n_2=n}\mu(n_1)\log n_2 = \frac{d}{d\gamma} \sum_{n_1n_2=n}\mu(n_1)n_2^{\gamma} |_{\gamma=0} . 
$
Then unfolding the coefficients $r(n,x)$ and applying this gives
\begin{align*}
S_{12}
            & = \sum_{\substack{\ell n_1\cdots n_a\\=n_{a+1}\cdots n_{2a}\\ n_j\leqslant x}} \frac{\Lambda(\ell)\tau_{-1/b}(n_1)\cdots \tau_{-1/b}(n_{2a})}{ (\ell n_1\cdots n_{2a})^{1/2}} \cdot \psi(n_1)\cdots\psi(n_{2a}) \\
            & = \frac{d}{d\gamma} \bigg( \sum_{\substack{\ell_1\ell_2 n_1\cdots n_a\\=n_{a+1}\cdots n_{2a}\\ n_j\leqslant x}}  \frac{\mu(\ell_1)  \tau_{-1/b}(n_1) \cdots \tau_{-1/b}(n_{2a}) }{ \ell_2^{1/2-\gamma}(\ell_1n_1\cdots n_{2a})^{1/2}} \cdot \psi(n_1)\cdots\psi(n_{2a})  \bigg)\bigg|_{\gamma=0}.
\end{align*}


As before, we apply Mellin inversion \eqref{mellin} to find
\begin{align*}
S_{12} =
&\frac{d}{d\gamma} \frac{B!^{2a}}{(2\pi i )^{2a}(\log x)^{2aB}}\int_{(c)^{2a}} 
 \sum_{\substack{\ell_1\ell_2 n_1\cdots n_a\\=n_{a+1}\cdots n_{2a}}} 
 \frac{ \mu(\ell_1)\tau_{-1/b}(n_1) \cdots \tau_{-1/b}(n_{2a}) }{ \ell_1^{1/2}\ell_2^{1/2-\gamma}n_1^{\frac{1}{2}+s_1} \cdots n_{2a}^{\frac{1}{2}+s_{2a}} }  
 \prod_{j=1}^{2a}x^{s_j}\frac{ds_j}{s_j^{B+1}} \bigg|_{\gamma=0}.
\end{align*}
Now a short calculation shows that the Dirichlet series in the integrand is given by 
\[
\mc{B}(\ul s,\gamma)\frac{ \prod_{i,j=1}^a \zeta(1+s_i+s_{a+j})^{1/b^2}\prod_{j=1}^a \zeta(1+s_{a+j})^{1/b} }{ \prod_{j=1}^a \zeta(1+s_{a+j}-\gamma)^{1/b} }
\]
where 
\begin{align*}
\mc{B}(\ul s, \gamma) = &\prod_p\frac{\prod_{i,j=1}^a\left( 1 - \frac{ 1 }{ p^{1+s_i+s_{a+j}} } \right)^{1/b^2}\prod_{j=1}^a\left( 1 - \frac{ 1 }{ p^{1+s_{a+j}} } \right)^{1/b} }{\prod_{j=1}^a\left( 1 - \frac{ 1 }{ p^{1+s_{a+j}-\gamma} } \right)^{1/b} }\\
                               &\qquad\qquad\qquad\qquad\qquad \times\sum_{ \substack{ \ell_1+\ell_2+m_1+\cdots+m_a\\= m_{a+1}+\cdots+m_{2a}\\m_j\geqslant 0 } } \frac{ \mu(p^{\ell_1})\tau_{-1/b}(p^{m_1})\cdots \tau_{-1/b}(p^{m_{2a}}) }{ p^{ \ell_1+\ell_2-\gamma+m_1(\frac{1}{2}+s_1 )+\cdots+m_{2a}( \frac{1}{2}+s_{2a} ) } }
\end{align*}
is the corresponding holomorphic factor. Again, this is easily seen to be holomorphic in the region $\sigma_j>-1/4$, $j=1,\ldots, 2a$. Then, taking the derivative inside the integral we get 
\begin{align*}
S_{12} = & \frac{B!^{2a}}{(2\pi i)^{2a}(\log x)^{2aB}}\int_{(c)^{2a}}
 \bigg[\mc{B}'(\ul{s},0)\prod_{i,j=1}^a \zeta(1+s_i+s_{a+j})^{1/b^2} \\
               &\qquad\qquad\qquad + \mc{B}(\ul s,0)\prod_{i,j=1}^a \zeta(1+s_i+s_{a+j})^{1/b^2}\frac{1}{b} \sum_{j=1}^a \frac{ \zeta'(1+s_{a+j}) }{ \zeta(1+s_{a+j}) }  
               \bigg]
                \prod_{\ell=1}^{2a}x^{s_\ell}\frac{ds_\ell}{s_\ell^{B+1}}.
\end{align*}

Now, the first integral can be treated as in the previous subsection (the only difference being the arithmetic factor $\mc{B}^\prime(\ul{s},0)$ which is of no real consequence). In this way we find it is $O((\log x)^{a^2/b^2})$. In the remaining integral we first note that $\mc{B}(\ul s,0)=\mc{A}(\ul s)$ and then let $s_j\mapsto s_j/\log x$ for each $j$ to give  
\begin{multline*}
S_{12} =  
 \frac{B!^{2a}}{(2\pi i)^{2a}}\int_{(1)^{2a}}\mc{A}(\frac{\ul s}{\log x})\prod_{i,j=1}^a \zeta(1+\frac{s_i+s_{a+j}}{\log x} )^{1/b^2}\frac{1}{b} \sum_{j=1}^a \frac{ \zeta'(1+\frac{s_{a+j}}{\log x}) }{ \zeta(1+\frac{s_{a+j}}{\log x}) }   \prod_{j=1}^{2a}e^{s_j}\frac{ds_j}{s_j^{B+1}}
 \\
+ O\big((\log x)^{a^2/b^2}\big).
\end{multline*}
As before we may trivially bound the integrand, this time by $\ll (\log x)^{a^2/b^2+1}$, and then truncate the integrals at height $t_j=\pm\sqrt{\log x}$ to give
\begin{multline*}
S_{12} =  
 \frac{B!^{2a}}{(2\pi i)^{2a}}\int_{1-i\sqrt{\log x}}^{1+i\sqrt{\log x}}\!\!\!\cdots \int_{1-i\sqrt{\log x}}^{1+i\sqrt{\log x}}
 \mc{A}(\frac{\ul s}{\log x})
 \\
 \times
 \prod_{i,j=1}^a \zeta(1+\frac{s_i+s_{a+j}}{\log x} )^{1/b^2}\frac{1}{b} \sum_{j=1}^a \frac{ \zeta'(1+\frac{s_{a+j}}{\log x}) }{ \zeta(1+\frac{s_{a+j}}{\log x}) }   
 \prod_{j=1}^{2a}e^{s_j}\frac{ds_j}{s_j^{B+1}}
+ O\big((\log x)^{a^2/b^2+1/2}\big)
\end{multline*}  
since the tail integrals result in an error $\ll (\log x)^{a^2/b^2+1-B/2}$ and $B\geqs 1$. Then, applying the Taylor and Laurent expansions given  in \eqref{A taylor} and \eqref{zeta laurent} along with 
\begin{equation}
\label{log der laurent}
 \frac{ \zeta'(1+\frac{s_{a+j}}{\log x}) }{ \zeta(1+\frac{s_{a+j}}{\log x}) } 
 =
 -\frac{\log x}{s_{a+j}}+O\Big(\frac{1}{\sqrt{\log x}}\Big),
\end{equation}
which is valid in the current region of integration, we find 
\begin{multline*}
S_{12} = -\frac{1}{b} \frac{\mc{A}(\ul{0})B!^{2a}(\log x)^{a^2/b^2+1}}{(2\pi i)^{2a}}  \int_{1-i\sqrt{\log x}}^{1+i\sqrt{\log x}}\!\!\! \cdots \int_{1-i\sqrt{\log x}}^{1+i\sqrt{\log x}} \prod_{i,j=1}^a\frac{1}{(s_i+s_{a+j})^{1/b^2}} \\
 \times \sum_{j=1}^a \frac{1}{s_{a+j}}\prod_{j=1}^{2a}e^{s_j}\frac{ds_j}{s_j^{B+1}}
+ O\big((\log x)^{a^2/b^2+1/2}\big).
 \end{multline*}
Extending the integrals back to $\pm i\infty$ incurs an error of size $O((\log x)^{a^2/b^2+1/2})$. Also, by symmetry the sum $\sum_{j=1}^a s_{a+j}^{-1}$ results in $a$-copies of the integral with a factor of $s_{2a}^{-1}$, say. Hence, we acquire the asymptotic formula
\begin{equation}
\label{S_12 asymp}
S_{12}= -\frac{a}{b}\mc{A}(\ul{0})\gamma(a,b)(\log x)^{a^2/b^2+1} + O \big( (\log x)^{a^2/b^2+1/2}\big)
\end{equation}
where
\begin{align*}
\gamma(a,b) = & \frac{B!^{2a}}{(2\pi i)^{2a}} \int_{(1)^{2a}} \prod_{i,j=1}^a\frac{1}{(s_i+s_{a+j})^{1/b^2}}\bigg[\prod_{j=1}^{2a-1}e^{s_j}\frac{ds_j}{s_j^{B+1}}\bigg]
e^{s_{2a}}\frac{ds_{2a}}{s_{2a}^{B+2}}.
\end{align*}

\subsection{Computation of the constants}\label{constants sec}

Applying \eqref{S_11 asymp} and \eqref{S_12 asymp} in Ng's formula \eqref{Ngs} we find that 
\begin{multline*}
S_1=\mc{A}(\ul{0})\frac{\beta(a,b)}{2}T(\log T)(\log x)^{a^2/b^2}+(a/b)\mc{A}(\ul{0})\gamma(a,b)T(\log x)^{a^2/b^2+1}
\\+O \big( T(\log x)^{a^2/b^2+1/2}\big).
\end{multline*}
Since $x=T^{\theta/(a+b)}$ it remains to show that the constants $\mc{A}(\ul{0})$, $\beta(a,b)$ and $\gamma(a,b)$ are positive.

 From the definition of $\mc{A}(\ul{s})$ given in \eqref{A(s)} a short calculation shows that 
\[
\mc{A}(\ul{0})=\prod_p \bigg(1-\frac{1}{p}\bigg)^{k^2}\sum_{m\geqs 0}\frac{\tau_{-k}(p^m)^2}{p^m}
\]
which is an absolutely convergent product. For the combinatorial constants we follow \cite[Lemma 8]{BH}. Recall that  
\[
\beta(a,b)=
 \frac{B!^{2a} }{(2\pi i )^{2a}} \int_{(1)^{2a}} \prod_{i,j=1}^a\frac{1}{(s_i+s_{a+j})^{1/b^2}} 
 \prod_{\ell=1}^{2a}e^{s_\ell}\frac{ds_\ell}{s_\ell^{B+1}}.
\]
For each term in the double product write 
\[
\frac{1}{(s_i+s_{a+j})^{1/b^2}} = \frac{1}{\Gamma(1/b^2)}\int_0^\infty e^{-(s_i+s_{a+j})x_{ij}}x_{ij}^{1/b^2}\frac{dx_{ij}}{x_{ij}}
\]
so that
\begin{align*}
\beta(a,b) = &\frac{B!^{2a}}{\Gamma(1/b^2)^{a^2} }\frac{1}{(2\pi i)^{2a}}\int_{(1)^{2a}} \int_{[0,\infty]^{a^2}} \Big[\prod_{i=1}^a e^{s_i(1-\sum_{j=1}^a x_{ij} )}\Big]\\
                   &\qquad\qquad\times \Big[  \prod_{j=1}^a e^{s_{a+j}(1-\sum_{i=1}^a x_{ij} )} \Big] \prod_{i,j=1 }^a x_{ij}^{1/b^2} \frac{ dx_{ij} }{ x_{ij} }\prod_{\substack{j=1}}^{2a}\frac{ds_j}{s_j^{B+1}}.
\end{align*}
After interchanging the order of integration and using the formula
\[
\frac{B!}{2\pi i}\int_{(c)} e^{s(1-X)}\frac{ds}{s^{B+1}} = \begin{cases}
      (1-X)^B & \text{ if } X\leqslant 1, \\
      0       & \text{ otherwise};
    \end{cases}
\]
we get
\begin{align*}
\beta(a,b) = & \frac{1}{\Gamma(1/b^2)^{a^2}}\int_{\mathscr{P}_{a,b}}\prod_{i=1}^a \bigg(1-\sum_{j=1}^a x_{ij}  \bigg)^B\prod_{\substack{j=1}}^a \bigg(1-\sum_{i=1}^a x_{ij}  \bigg)^B \prod_{i,j=1 }^a x_{ij}^{1/b^2} \frac{ dx_{ij} }{ x_{ij} }
\end{align*}
where
\[
\mathscr P_{a,b} = \Big\{ (x_{ij})\in\mathbb R^{a^2}:\  x_{ij}\geqslant 0,\  \sum_{i=1}^a x_{ij}\leqslant 1,\   \sum_{j=1}^a x_{ij}\leqslant 1  \Big\}.
\]
A similar formula holds for $\gamma(a,b)$; the only difference being that the factor $(1-\sum_i x_{ia})^B$ is replaced by $(1-\sum_i x_{ia})^{B+1}/(B+1)$. From this we easily see that it is also positive. Proposition \ref{S_1 prop} then follows.


\section{Proof of proposition 2 }\label{S2Proof}
In this section we shall prove Proposition \ref{S_2 prop} assuming Theorem \ref{twisted moment thm}. We start from the formula
\[
S_2=\frac{d}{d\alpha}\frac{d}{d\beta}\sum_{0\leqs \gamma \leqs T}\zeta(\rho+\alpha)\zeta(1-\rho+\beta)Q(\rho)Q(1-\rho) \big|_{\alpha=\beta=0}
\]
where 
\[
Q(s) =  \bigg(\sum_{n\leqs x}\frac{\tau_{-1/b}(n)\psi(n)}{n^s}\bigg)^{a+b}
        = \sum_{n\leqs y}\frac{a(n)}{n^s}
\]
with
\begin{equation}\label{coefficients}
a(n)=a(n,x) = \sum_{\substack{n_1\cdots n_{a+b}=n\\ n_j\leqs x}}\prod_{j=1}^{a+b} \tau_{-1/b}(n_j)\psi(n_j)
\end{equation}
and $y=x^{a+b}=T^\theta$ with $\theta<1/2$.
Thereom \ref{twisted moment thm} then gives $S_2$ as a sum of three terms, two of which involve the $Z$ terms. We write this as 
\[
S_2=S_{21}+S_{22}+S_{23}
\]
with
\begin{multline}
\label{S_21}
S_{21} = \frac{d}{d\alpha}\frac{d}{d\beta}  
              \sum_{g\leqs y}\sum_{\substack{h,k\leqs y/g\\(h,k)=1}}\frac{a(gh)a(gk)}{ghk}
              \frac{1}{2\pi}\int_1^T\log\bigg(\frac{t}{2\pi}\bigg)
    \bigg[
              \frac{\zeta(1+\alpha+\beta)}{h^{\beta}k^{\alpha}}
 \\     + \bigg(\frac{t}{2\pi}\bigg)^{-\alpha-\beta}
             \frac{\zeta(1-\alpha-\beta)}{h^{-\alpha}k^{-\beta}}
    \bigg]dt
\bigg|_{\alpha=\beta=0}
\end{multline}
and
\begin{multline}
\label{S_22}
S_{22} = \frac{d}{d\alpha}\frac{d}{d\beta}  \frac{d}{d\gamma}  
              \sum_{g\leqs y}\sum_{\substack{h,k\leqs y/g\\(h,k)=1}}\frac{a(gh)a(gk)}{ghk}
              \frac{1}{2\pi} \int_1^T\bigg[ 
               Z_{\alpha,\beta,\gamma,h,k}
\\       + \bigg(\frac{t}{2\pi}\bigg)^{-\alpha-\beta}Z_{-\beta,-\alpha,\gamma,h,k}
      + \bigg(\frac{t}{2\pi}\bigg)^{-\beta-\gamma}Z_{\alpha,-\gamma,-\beta,h,k}
     \bigg]dt
     \bigg|_{\alpha=\beta=\gamma=0}.
\end{multline}
The formula for $S_{23}$ has $\alpha$ and $\beta$ interchanged in the integrand but by symmetry this is simply equal to $S_{22}$. We compute $S_{21}$ first.

\subsection{Computing $S_{21}$}

Our aim is to show $S_{21}\ll T(\log T)^{ k^2+3 }$. It is helpful to express the integrand in a form in which the holomorphy is immediately visible, as in \cite{CFKRS}. For this purpose we use the formula
\begin{multline}\label{contour int rep}
\frac{\zeta(1+\alpha+\beta)}{h^{\beta}k^{\alpha}}
      + \bigg(\frac{t}{2\pi}\bigg)^{-\alpha-\beta}
             \frac{\zeta(1-\alpha-\beta)}{h^{-\alpha}k^{-\beta}} \\
=-\Big( \frac{t}{2\pi} \Big)^{ -\frac{\alpha+\beta}{2} }  
    \frac{1}{ (2\pi i )^2 } \oint\!\!\!\oint \bigg(\frac{t}{2\pi}\bigg)^{\frac{z_2-z_1}{2}}\frac{ \zeta(1+z_1-z_2)(z_2-z_1)^2 }{h^{-z_2}k^{z_1}\prod_{i=1}^2(z_i-\alpha)(z_i+\beta)}dz_1dz_2
\end{multline}
where the integrals are over circles of radii $\ll 1/\log T$ that enclose the shifts $\alpha$ and $\beta$. This formula follows from a short residue computation. Interchanging the sum and integral we get 
\begin{multline*}
S_{21} = -  \frac{d}{d\alpha}\frac{d}{d\beta}  
                \frac{1}{2\pi}\int_0^T 
                  \Big( \frac{t}{2\pi} \Big)^{ -\frac{\alpha+\beta}{2} }  
                  \log\bigg(\frac{t}{2\pi}\bigg)
          \\
                \times
                  \frac{1}{ (2\pi i )^2 } \oint\!\!\!\oint \bigg(\frac{t}{2\pi}\bigg)^{\frac{z_2-z_1}{2}}
                  \frac{ \zeta(1+z_1-z_2)(z_2-z_1)^2 }{\prod_{i=1}^2(z_i-\alpha)(z_i+\beta)}      
                  F(x,z_1,z_2)
                dz_1dz_2    dt
               \bigg|_{\alpha=\beta=0}
\end{multline*}
where 
\[
F(x,z_1,z_2) = \sum_{g\leqs y}\sum_{\substack{h,k\leqs y/g\\(h,k)=1}}\frac{a(gh)a(gk)}{gh^{1-z_2}k^{1+z_1}}
                     = \sum_{\substack{h,k\leqs y}}\frac{a(h)a(k)(h,k)^{1+z_1-z_2}}{h^{1-z_2}k^{1+z_1}}.
\]
We now perform the differentiation and estimate the contour integrals. The contour lengths and the factor of $(z_2-z_1)^2$ contribute $\ll(\log T)^{-4}$ whilst the zeta function, negative powers of $z_i$ and $\log(t/2\pi )$ term contribute $\ll (\log T)^6$. The differentiation gives us two extra factors of $\log T$ and thus in total we get 
\[
S_{21}\ll T(\log T)^4 \max_{|z_1|,|z_2|\ll 1/\log T}|F(x,z_1,z_2)|.
\]
We are therefore required to show that 
\[
\max_{|z_1|,|z_2|\ll 1/\log T}|F(x,z_1,z_2)|\ll (\log T)^{\frac{a^2}{b^2}-1}.
\]
We will show this by the methods of the previous section. Throughout the following we shall assume that $z_1,z_2\ll 1/\log T$.

We proceed by first unfolding the sum using the formula for the coefficients $a(n)$ given in \eqref{coefficients}. Writing $a+b=N$ for short we find 
\[
F(x,z_1,z_2) =\!\!\!\!\! \sum_{\substack{h_1,\cdots, h_N\leqs x\\k_1.\cdots,k_N\leqs x}}
                         \frac{\Big[\prod_{i=1}^N \tau_{-1/b}(h_i)\tau_{-1/b}(k_i)\psi(h_i)\psi(k_i)\Big]( h_1\cdots h_N,k_1\cdots k_N)^{1+z_1-z_2}}{(h_1\cdots h_N)^{1-z_2}(k_1\cdots k_N)^{1+z_1}}
\]
By the Mellin inversion formula \eqref{mellin} we acquire
\begin{equation*}
F(x,z_1,z_2) = 
\frac{B!^{2N}}{(\log x)^{2NB}}\frac{1}{(2\pi i)^{2N}}\int_{(c)^{2N}}
\mc{F}_{z_1,z_2}(\underline{s})
                        \prod_{j=1}^{2N}x^{s_j}\frac{ds_j}{s_j^{B+1}}
\end{equation*}
where
\[
\mc{F}_{z_1,z_2}(\underline{s})= \sum_{\substack{h_1,\cdots, h_N\geqs 1\\k_1.\cdots,k_N\geqs 1}}
                         \frac{\Big[\prod_{i=1}^N \tau_{-1/b}(h_i)\tau_{-1/b}(k_i)\Big]( h_1\cdots h_N,k_1\cdots k_N)^{1+z_1-z_2}}
                        {h_1^{1+s_1-z_2}\cdots h_N^{1+s_N-z_2}k_1^{1+s_{N+1}+z_1}\cdots k_N^{1+s_{2N}+z_1}}
\]
and $c= 1/\log x$. Now, a short computation with Euler products shows that 
\begin{equation}\label{F euler}
\mc{F}_{z_1,z_2}(\underline{s})=\mc{C}_{z_1,z_2}(\underline{s})
\frac{\prod_{i,j=1}^N\zeta(1+s_i+s_{j+N})^{1/b^2}}
{\prod_{j=1}^{N}\zeta(1+s_j-z_2)^{1/b}\zeta(1+s_{j+N}+z_2)^{1/b}}
\end{equation}
where $\mc{C}_{z_1,z_2}(\underline{s})$ is an absolutely convergent Euler product in the region $\sigma_j>-1/4$, $j=1,\ldots, N$. This gives us the trivial bound 
\begin{equation}\label{F bound}
\qquad\qquad \qquad \qquad\qquad\mc{F}_{z_1,z_2}(\underline{s})\ll (\log x)^{\frac{N^2}{b^2}+\frac{2N}{b}}, \qquad\qquad\qquad \Re(s_j)\asymp\frac{1}{\log x}.
\end{equation}

Substituting $s_j\mapsto s_j/\log x$ for each $j$ gives
\[
F(x,z_1,z_2) = 
\frac{B!^{2N}}{(2\pi i)^{2N}}\int_{(1)^{2N}}
\mc{F}_{z_1,z_2}(\underline{s}/\log x)
                        \prod_{j=1}^{2N}e^{s_j}\frac{ds_j}{s_j^{B+1}}.
\]
By \eqref{F bound}, any given tail integral over the line from $1+i\sqrt{\log x}$ to $1+i\infty$ results in a total contribution of 
\[
(\log T)^{\frac{N^2}{b^2}+\frac{2N}{b}}\int_{1+i\sqrt{\log x}}^{1+i\infty} \frac{ds}{|s|^{B+1}}\ll (\log T)^{\frac{N^2}{b^2}+\frac{2N}{b}-B/2}=(\log T)^{k^2+4k+3-B/2}
\]
since the other integrals are absolutely convergent.
Therefore, on taking $B\geqs 10+8k$ this term is $\ll (\log T)^{k^2-2}$. Consequently, we may localise the integral:
\[
F(x,z_1,z_2) = 
\frac{B!^{2N}}{(2\pi i)^{2N}}\int_{1-i\sqrt{\log x}}^{1+i\sqrt{\log x}}\!\!\!\!\cdots \int_{1-i\sqrt{\log x}}^{1+i\sqrt{\log x}}
\mc{F}_{z_1,z_2}(\underline{s}/\log x)
                        \prod_{j=1}^{2N}e^{s_j}\frac{ds_j}{s_j^{B+1}}+O\big((\log T)^{k^2-2}\big).
\]
In this new region of integration we have the bounds 
\ben
\label{inverse zeta bound}
\zeta\Big(1+\frac{s_j}{\log x}\pm z_i\Big)^{-1}\ll \frac{|s_j|}{\log x}+|z_j|\ll \frac{1}{\log T}(|s_j|+1)
\een
and 
\ben\label{zeta bound}
\zeta(1+(s_i+s_j)/\log x)\ll \frac{\log T}{|s_i+s_j|}
\een
as well as $\mc{C}_{z_1,z_2}(\underline{s}/\log x)\ll 1$. Applying these bounds in \eqref{F euler} we find that in the region of integration we have 
\begin{equation}
\label{mcF bound}
\mc{F}_{z_1,z_2}(\ul{s}/\log x)\ll (\log T)^{\frac{N^2}{b^2}-\frac{2N}{b}}
\frac{\prod_{j=1}^{N}(|s_j|+1)^{1/b}(|s_{j+N}|+1)^{1/b}}{\prod_{i,j=1}^N|s_i+s_{j+N}|^{1/b^2}}.
\end{equation}
By the absolute convergence of the integrals we get
\[
F(x,z_1,z_2)\ll (\log T)^{\frac{N^2}{b^2}-\frac{2N}{b}}=(\log T)^{k^2-1}
\]
and the required bound for $S_{21}$ follows.

\subsection{Computing $S_{22}$} The bound $S_{22}\ll T(\log T)^{k^2+3}$ follows similarly with a few extra technicalities. We begin with an equivalent version of formula \eqref{contour int rep} for the sum of $Z$ terms. This is given by 
\begin{multline}
\label{contour int rep 2}
Z_{\alpha,\beta,\gamma,h,k}+\bigg(\frac{t}{2\pi }\bigg)^{-\alpha-\beta}Z_{-\beta,-\alpha,\gamma,h,k}+\bigg(\frac{t}{2\pi }\bigg)^{-\beta-\gamma}Z_{\alpha,-\gamma,-\beta,h,k}
\\ =
-\frac{1}{2}\bigg(\frac{t}{2\pi}\bigg)^{-\frac{\alpha+\beta+\gamma}{2}}
\frac{1}{(2\pi i)^3}\oint\oint\oint 
\frac{Z_{z_1,-z_2,z_3,h,k}\Delta(z_1,z_2,z_3)^2}{\prod_{i=1}^3(z_i-\alpha)(z_i+\beta)(z_i-\gamma)}
\\
\times
\bigg(\frac{t}{2\pi}\bigg)^{\frac{z_1-z_2+z_3}{2}}dz_1dz_2dz_3
\end{multline}
where $\Delta(z_1,z_2,z_3)=\prod_{i<j}(z_j-z_i)$ is the vandermonde determinant and the integrals are over circles of radii $\ll 1/\log T$ enclosing the shifts.
Again, this formula follows by a simple (but tedious) residue calculation. As before, we plan to interchange the order of summation and integration and then compute the resulting sum. The computations are considerably simplified if we write $Z_{z_1,-z_2,z_3,h,k}$  in a diagonal form.

\begin{lem}\label{diag lem} Let $H(s)$ be an analytic function which satisifes $H(0)=1$ and is zero at $2s=z_2-z_1$ and $2s=z_2-z_3$. Furthermore, suppose that $H(u+iv)\ll_u (\log T)^2 e^{-cv^2}$ for fixed $u$ and large $v$. For $c>0$ let 
\[
\tilde{Z}_{z_1,z_2,z_3,h,k}(T)
=
\frac{1}{2\pi i }\int_{c-i\infty}^{c+i\infty}\frac{H(s)}{s}
T^{2s}
\sum_{hm=kn}\frac{f_{z_1,z_3}(m) n^{-z_2}}{(hk)^{1/2}(mn)^{1/2+s}} ds
\] 
where $f_{z_1,z_2}(n)$ is given by \eqref{f}. Then for $h,k\leqs T$ with $(h,k)=1$ and $z_j\ll 1/\log T$ we have 
\begin{equation}\label{Z diag}
Z_{z_1,z_2,z_3,h,k}=hk\tilde{Z}_{z_1,z_2,z_3,h,k}(T)+O\Big(\tau(k)\Big(\frac{T^2}{hk}\Big)^{-\frac{1}{\log_3 T}}(\log T)^5(\log_3 T)\Big)
+O\Big(\frac{\tau(k)}{(\log T)^{A}}\Big)
\end{equation}
where $A>0$ is an arbitrary constant.
\end{lem}
\begin{proof}
This is similar to the proof of Theorem 1.2 in \cite{BBLR}. Since $(h,k)=1$ we have
\[
\tilde{Z}_{z_1,z_2,z_3,h,k}(T)=\frac{1}{h^{z_2}}\frac{1}{2\pi i }\int_{c-i\infty}^{c+i\infty}\frac{H(s)}{s}T^{2s}\frac{1}{(hk)^{1+s}}\sum_{\ell }\frac{f_{z_1,z_3}(k\ell) }{\ell^{1+z_2+2s}} ds.
\] 
Now, a short computation shows that 
\begin{multline*}
     \sum_{\ell }\frac{f_{z_1,z_3}(k\ell) }{\ell^{1+z_2+2s}}
  = 
     \frac{\zeta(1+z_1+z_2+2s)\zeta(1+z_2+z_3+2s)}{\zeta(1+z_2+2s)}     
  \\
     \times
     \prod_{p^{k_p} || k}\frac{\sum_{m\geqs 0}f_{z_1,z_3}(p^{m+k_p})p^{-m(1+z_2+2s)}}{\sum_{m\geqs 0}f_{z_1,z_3}(p^{m})p^{-m(1+z_2+2s)}}.
\end{multline*}
This is holomorphic for $\sigma>0$ so we may freely shift to the line $\Re(s)=1/\log T$. We then truncate the integral at height $t=\pm\sqrt{A\log\log T}$ incurring an error of size
\begin{equation}
\label{trunc est}
\ll \frac{1}{hk} e^{-cA\log\log T}(\log T)^5 \tau(k)\ll \frac{\tau(k)}{hk}(\log T)^{-A^\prime};
\end{equation}
the factor of $\tau(k)$ owing to the fact that uniformly for $\sigma\geqs -1/4$, say, the product over primes dividing $k$ is 
\[
\ll \prod_{p|k}\Big(1+O\big(p^{-1/2}\big)\Big)\ll \tau(k).
\]
Then, shifting the contour to $\Re(s)=-1/\log_3 T$ we remain in the zero-free region of $\zeta(1+s)$ and hence find only a simple pole at $s=0$ (since the zeros of $H(s)$ cancel the other poles). Since $\zeta(s)^{\pm 1} \ll \log T$ on the contour, the integral over the left edge is 
\[
\ll T^{-2/\log_3 T}(hk)^{-1+1/\log_3 T}\tau(k) (\log T)^5(\log_3 T)
\]
whilst the horizontal integrals give a lower order contribution plus a contribution of size \eqref{trunc est}. Therefore,
\begin{multline*}
\tilde{Z}_{z_1,z_2,z_3,h,k}(T)=\frac{1}{hk}Z_{z_1,z_2,z_3,h,k}+O\Big(\frac{\tau(k)}{hk}\Big(\frac{T^2}{hk}\Big)^{-\frac{1}{\log_3 T}}(\log T)^5(\log_3 T)\Big)
\\
+O\Big(\frac{\tau(k)}{hk}(\log T)^{-A^\prime}\Big)
\end{multline*}
and the result follows.
\end{proof}


\begin{rem}
The function $H(s)$ in the lemma is typical and can be prescribed a more precise form. For example, one can take $H(s)=Q_{\ul{z}}(s)\exp(s^2)$ where $Q_{\ul{z}}(s)$ is a polynomial which is 1 at $s=0$ and zero at $2s=z_2-z_1$ and $2s=z_2-z_3$.
\end{rem}
The first error term of \eqref{Z diag} gives a total contribution to $S_{22}$ of
\[
\ll\Big(\frac{T^2}{y^2}\Big)^{-\frac{1}{\log_3 T}}(\log T)^5(\log_3 T)\sum_{g\leqs y}\sum_{\substack{h,k\leqs y/g\\(h,k)=1}}\frac{|a(gh)||a(gk)|\tau(k)}{ghk}\int_1^T dt\ll T^{1-\frac{1}{\log_3 T}}(\log T)^C
\]
since $y\leqs T^{1/2}$ and our coefficients satisfy divisor function bounds. The second error term contributes $\ll T/(\log T)^{A^\prime}$. As noted earlier, when estimating these error terms we can perform the differentiation via Cauchy's formula over circles of radii $1/\log T$. This adds three extra factors of $\log T$ which is tolerable.

Applying Lemma \ref{diag lem} and \eqref{contour int rep 2} in \eqref{S_22} we find that 
\begin{multline*}
S_{22} = -\frac{1}{4\pi}\frac{d}{d\alpha}\frac{d}{d\beta}  \frac{d}{d\gamma} 
                 \bigg(\frac{t}{2\pi}\bigg)^{-\frac{\alpha+\beta+\gamma}{2}}
                 \int_1^T
                 \frac{1}{(2\pi i)^3}\oint\!\!\!\oint\!\!\!\oint 
                 G(x,\underline{z})
        \\      \qquad   \times 
                 \frac{\Delta(z_1,z_2,z_3)^2}{\prod_{i=1}^3(z_i-\alpha)(z_i+\beta)(z_i-\gamma)}
                 \bigg(\frac{t}{2\pi}\bigg)^{\frac{z_1-z_2+z_3}{2}}dz_1dz_2dz_3dt
                 \bigg|_{\alpha=\beta=\gamma=0}
+o(T)
\end{multline*}
where 
\[
G(x,\underline{z}) 
=
 \sum_{g\leqs y}\sum_{\substack{h,k\leqs y/g\\(h,k)=1}}\frac{a(gh)a(gk)}{g}\tilde{Z}_{z_1,-z_2,z_3,h,k}(T).
\]
Performing the differentiation and then trivially estimating the $z_j$ integrals gives
\[
S_{22}\ll T (\log T)^{9+3-2\cdot 3-3}\max_{|z_j|\ll1/\log T}|G(x,\ul{z})|.
\]
Thus, we are required to show that 
\begin{equation}
\label{G bound}
\max_{|z_j|\ll1/\log T}|G(x,\ul{z})|\ll (\log T)^{k^2}.
\end{equation}

Unfolding the integral we have
\begin{align*}
G(x,\ul{z})
=  &
\frac{1}{2\pi i }\int_{c-i\infty}^{c+i\infty}
\frac{H(w)}{w}
T^{2w}
 \sum_{g\leqs y}\sum_{\substack{h,k\leqs y/g\\(h,k)=1}}\frac{a(gh)a(gk)}{g}\sum_{hm=kn}\frac{f_{z_1,z_3}(m) n^{z_2}}{(hk)^{1/2}(mn)^{1/2+w}} dw
\\
=  &
\frac{1}{2\pi i }\int_{c-i\infty}^{c+i\infty}
\frac{H(w)}{w}
T^{2w}
 \sum_{hm=kn}\frac{a(h)a(k)f_{z_1,z_3}(m) n^{z_2}}{(hk)^{1/2}(mn)^{1/2+w}} dw.
\end{align*}
Applying the definitions of the coefficients $a(n)$ in \eqref{coefficients} and $f_{z_1,z_2}(n)$ in \eqref{f} along with the Mellin inversion formula \eqref{mellin} for the weights $\psi(n)$ gives 
\begin{align*}
 \sum_{hm=kn}&\frac{a(h)a(k)f_{z_1,z_3}(m) n^{z_2}}{(hk)^{1/2}(mn)^{1/2+w}} 
 = \frac{B!^{2N}}{(\log x)^{2NB}(2\pi i)^{2N}}\int_{(\kappa)^{2N}}
       \mc{G}(w,\underline{s},\underline{z})\prod_{j=1}^{2N}x^{s_j}\frac{ds_j}{s_j^{B+1}}
\end{align*}
where
\begin{equation*}
\begin{split}
\label{mc G}
\mc{G}(w,\underline{s}, &\underline{z})
=
\!\!\!\!\!\!\!\!\!\sum_{\substack{h_1\cdots h_Nm_1m_2m_3 \\ = k_1\cdots k_Nn}}
  \frac{\Big(\prod_{i=1}^N \tau_{-1/b}(h_i)\tau_{-1/b}(k_i)\Big)\mu(m_1)}
          {\Big(\prod_{i=1}^{N}h_i^{\frac{1}{2}+s_i}k_i^{\frac{1}{2}+s_{i+N}}\Big)
          m_1^{\frac{1}{2}+w}m_2^{\frac{1}{2}+z_1+w}m_3^{\frac{1}{2}+z_3+w}n^{\frac{1}{2}-z_2+w}}
\end{split}
\end{equation*}
and $N=a+b$, as before. By considering its Euler product we find
\begin{equation}
\begin{split}
\label{mc G prod}
         &  \mc{G}(w,\underline{s}, \underline{z})
     =   
             \mc{D}(w,\ul{s},\ul{z})
             \frac{\zeta(1+z_1-z_2+2w)\zeta(1-z_2+z_3+2w)}{\zeta(1-z_2+2w)}
   \\  \times          
         &  \frac{\prod_{i,j=1}^N\zeta(1+s_i+s_{j+N})^{1/b^2}\prod_{j=1}^N\zeta(1+s_{j+N}+w)^{1/b}}
             {\prod_{i=1}^N\zeta(1+s_i-z_2+w)^{1/b}\prod_{j=1}^N\zeta(1+s_{j+N}+z_1+w)^{1/b}\zeta(1+s_{j+N}+z_3+w)^{1/b}}
\end{split}
\end{equation}
where $\mc{D}(w,\ul{s},\ul{z})$ is an absolutely convergent Euler product provided $\sigma_j>-1/4$, $j=1,\ldots 2N$, $w>-1/4$.

The integral over $w$ has now done its job and we can shift to the left picking up the pole at $w=0$. We first set $\kappa=1/4$
and note this implies $\prod_{j=1}^{2N}x^{s_j}\ll x^{N/2}=T^{\theta/2}$. Shifting the $w$ contour to the line with $\Re(w)=-1/4+\epsilon$ and assuming RH we only encounter a simple pole at $w=0$. By \eqref{mc G prod}, $\mc{G}(w,\ul{s},\ul{z})\ll(wT)^{\epsilon}$ and hence the integral over the new line gives a total contribution of 
\[
\ll T^{-1/2+\theta/2+\epsilon}=o(1).
\]
Thus we acquire 
\[
G(x,\ul{z})
=
\frac{B!^{2N}}{(\log x)^{2NB}(2\pi i)^{2N}}
\int_{(\kappa)^{2N}}
\mc{G}(0,\underline{s},\underline{z})\prod_{j=1}^{2N}x^{s_j}\frac{ds_j}{s_j^{B+1}}
+o(1).
\]

We now shift $\kappa$ to $1/\log x$ and  substitute $s_j\mapsto s_j/\log x$ for each $j$ to give 
\[
G(x,\ul{z})
=
\frac{B!^{2N}}{(2\pi i)^{2N}}
\int_{(1)^{2N}}
\mc{G}(0,\underline{s}/\log x,\underline{z})\prod_{j=1}^{2N}e^{s_j}\frac{ds_j}{s_j^{B+1}}
+o(1).
\]
By \eqref{mc G prod} we have the trivial estimate
\ben
\label{G trivial}
\qquad\qquad\qquad\mc{G}(0,\underline{s}/\log x,\underline{z}), 
\ll
(\log T)^{\frac{N^2}{b^2}+\frac{4N}{b}+1},
\qquad\qquad\Re(s_j)\asymp 1
\een
whilst for $|s_j|=o(\log x)$ we have 
\begin{multline}
\label{G laurent}
\mc{G}(0,\underline{s}/\log x,\underline{z})
\ll
(\log T)^{\frac{N^2}{b^2}-\frac{2N}{b}+1} 
\\
\times
\frac{\prod_{i=1}^N(1+|s_i|)^{1/b}\prod_{j=1}^N(1+|s_{j+N}|)^{1/b}(1+|s_{j+N}|)^{1/b}}
{\prod_{i,j=1}^N|s_i+s_{j+N}|^{1/b^2}\prod_{j=1}^N |s_{j+N}|^{1/b}}
\end{multline}
where we have used the bounds for $\zeta(s)$ given in \eqref{inverse zeta bound} and \eqref{zeta bound}. As before, truncating any of the $s_j$ integrals at height $t_j=\sqrt{\log x}$ leads to an error of 
\[
(\log T)^{\frac{N^2}{b^2}+\frac{4N}{b}+1}\int_{1+i\sqrt{\log x}}^{1+i\infty}\frac{ds}{|s|^{B+1}}
\ll
(\log T)^{k^2+6k+6-B/2} 
\]
and thus choosing \[B\geqs 12k+14\] this is $\ll (\log T)^{k^2-1}$. Performing this truncation in each variable gives
\[
G(x,\ul{z})
=
\frac{B!^{2N}}{(2\pi i)^{2N}}
\int_{1-i\sqrt{\log x}}^{1+i\sqrt{\log x}}\!\!\!\!\!\!\cdots\int_{1-i\sqrt{\log x}}^{1+i\sqrt{\log x}}
\mc{G}(0,\underline{s}/\log x,\underline{z})\prod_{j=1}^{2N}e^{s_j}\frac{ds_j}{s_j^{B+1}}
+O\big((\log T)^{k^2-1}\big)
\]
and by \eqref{G laurent} this is $\ll(\log T)^{\frac{N^2}{b^2}-\frac{2N}{b}+1} 
= (\log T)^{k^2}$. The bound $S_{22}\ll (\log T)^{k^2+3}$ then follows.


\section{Proof of Theorem \ref{twisted moment thm}}\label{Theorem2Proof}
Let
\begin{equation}
\label{S_3}
  S_3=S_3(\alpha,\beta,T)
        =   \sum_{0\leqs \gamma\leqs T}\zeta(\rho+\alpha)\zeta(1-\rho+\beta)Q(\rho)\overline{Q}(1-\rho)
\end{equation}
where $\alpha,\beta$ are small ($\ll1/\log T$), complex shifts. 
These types of mean values have been considered before by several authors \cite{CGG, CGGSimple, Ng1}. Accordingly, we shall only briefly describe the initial steps using \cite{CGG} as our main reference.

We write $S_3$ as the integral over the positively oriented rectangular contour $\Gamma$ with vertices $a+i,a+iT,1-a+iT,1-a+i$, $a=1+1/\log T$:
\[
S_3=\frac{1}{2\pi i }\int_\Gamma \frac{\zeta^\prime(s)}{\zeta(s)}\zeta(s+\alpha)\zeta(1-s+\beta)Q(s)\overline{Q}(1-s)ds. 
\]
Since $Q(s)\ll y^{1-\sigma+\epsilon}$, $\zeta(s)\ll t^{\tfrac{1}{2}(1-\sigma)+\epsilon}$ and $T$ can be chosen such that $(\zeta^\prime/\zeta)(s)\ll (\log T)^2$ on the contour, we find that the horizontal sections contribute $O(y^{1+\epsilon}T^{1/2+\epsilon})$. For the contour on the left hand side we apply the functional equation
\[
\frac{\zeta^\prime(s)}{\zeta(s)}=\frac{\chi^\prime(s)}{\chi(s)}-\frac{\zeta^\prime(1-s)}{\zeta(1-s)}\
\]
where $\chi(s)=\pi^{s-\tfrac{1}{2}}\Gamma((1-s)/2)/\Gamma(s/2)$ is the factor appearing in the functional equation $\zeta(s)=\chi(s)\zeta(1-s)$. The integral involving $-\zeta^\prime(1-s)/\zeta(1-s)$ is given by 
\begin{align*}
&-\frac{1}{2\pi i}\int_{1-a+iT}^{1-a+i}\frac{\zeta'(1-s)}{\zeta(1-s)}\zeta(s+\alpha)\zeta(1-s+\beta)Q(s)\ol{Q}(1-s)ds\\
& = \frac{1}{2\pi }\int_{1}^T\frac{\zeta'(a-it)}{\zeta(a-it)}\zeta(1-a+it+\alpha)\zeta(a-it+\beta)Q(1-a+it)\ol{Q}(a-it)dt\\
& = \overline{\frac{1}{2\pi } \int_1^T \frac{\zeta'(a+it)}{\zeta(a+it)}\zeta(1-a-it+\overline{\alpha})\zeta(a+it+\overline{\beta})\ol{Q}(1-a-it)Q(a+it)dt}\\
& = \overline{\frac{1}{2\pi i}\int_{a+i}^{a+iT}\frac{\zeta'(s)}{\zeta(s)}\zeta(1-s+\overline{\alpha})\zeta(s+\overline{\beta})\ol{Q}(1-s)Q(s)ds}.
\end{align*}
This integral can therefore  be expressed in terms of the integral over the right edge of the contour.
For the integral involving $\chi^\prime(s)/\chi(s)$ we shift to the half-line and apply Stirling's formula in the form
\[
\frac{\chi^\prime(\tfrac{1}{2}+it)}{\chi(\tfrac{1}{2}+it)}=-\log\Big(\frac{|t|}{2\pi}\Big)+O(1/|t|),\qquad t\geqs 1.
\]  
In this way, we find that 
\[
S_3=I(\alpha, \beta,T)+\overline{I(\overline{\beta}, \overline{\alpha}, T)}+J(\alpha, \beta, T)+O(yT^{1/2+\epsilon})
\]
where
 \ben
\label{I}
I(\alpha, \beta,T)=\frac{1}{2\pi i}\int_{a+i}^{a+iT} \frac{\zeta^\prime(s)}{\zeta(s)}\zeta(s+\alpha)\zeta(1-s+\beta)Q(s)\overline{Q}(1-s)ds.
\een
and
\ben
\label{J}
J(\alpha, \beta, T)=\frac{1}{2\pi }\int_{1}^{T} \big[\log(t/2\pi)+O(1/t)]\zeta(\tfrac{1}{2}+\alpha+it)\zeta(\tfrac{1}{2}+\beta-it)|Q(\tfrac{1}{2}+it)|^2dt.
\een
We are thus required to show that $I(\alpha, \beta, T)=\mc{I}(\alpha,\beta,T)+O(T(\log T)^{-A})$ where $\mc{I}$ is given by \eqref{mc I}, and that 
$J(\alpha, \beta, T)=\mc{J}(\alpha, \beta, T)+O(T(\log T)^{-A})$ where $\mc{J}$ is given by \eqref{mc J}.

\subsection{Computing $J$}

The integral $J$ can be computed by integrating by parts and using well known formulas for the twisted second moment of the zeta function. In our case (with the shifts $\alpha,\beta$) these mean values have been considered by Pratt and Rubles \cite{PR}. After a slight rephrasing, their Theorem 1.1 states that
\begin{multline*}
        \int_{1}^{T} \zeta(\tfrac{1}{2}+\alpha+it)\zeta(\tfrac{1}{2}+\beta-it)|Q(\tfrac{1}{2}+it)|^2dt
 \\   = \sum_{g\leqs y}\sum_{\substack{h,k\leqs y/g\\(h,k)=1}}\frac{a(gh)\overline{a(gk)}}{g{hk}}
         \int_1^T 
         \bigg[
                 \frac{\zeta(1+\alpha+\beta)}{h^{\beta}k^{\alpha}}
              + \bigg(\frac{t}{2\pi}\bigg)^{-\alpha-\beta}
                 \frac{\zeta(1-\alpha-\beta)}{h^{-\alpha}k^{-\beta}}
         \bigg]dt
 \\     +O(T^{3/20}y^{33/20})+O(y^{1/2}T^{1/2+\epsilon}).
\end{multline*} 
Thus, after integrating by parts we see that $J$ is indeed given by $\mc{J}$ plus an acceptable error. 

\subsection{Computing $I$: Initial manipulations}
Instead of working directly with $I(\alpha,\beta,T)$ we work with  the integral 
\ben
\label{mc K}
\mc{K}=\mc{K}_{\alpha,\beta}(\gamma):=
\frac{1}{2\pi i}\int_{a+i}^{a+iT} \frac{\zeta(s+\gamma)}{\zeta(s)}\zeta(s+\alpha)\zeta(1-s+\beta)Q(s)\ol{Q}(1-s)ds
\een
so that 
\[
I(\alpha,\beta,T)=\frac{d}{d\gamma}\mc{K}(\gamma)\bigg|_{\gamma=0}.
\]
As with the other shifts we will assume throughout that $\gamma\ll1/\log T$ and derive our formula for $\mc{K}_{\alpha,\beta}(\gamma)$ with error terms uniform in $\gamma$. The differentiation can then be performed by applying Cauchy's formula over a circle of radius $\ll1/
\log T$. 

In \eqref{mc K} we apply the functional equation $\zeta(1-s+\beta)=\chi(s-\beta)\zeta(s-\beta)$ and then expand each term as a Dirichlet series to give
\be
\begin{split}
\mc{K} = & \sum_{m_1,m_2,m_3,m_4,h,k}\frac{\mu(m_1)a(h)\ol{a(k)}}{m_2^\gamma m_3^\alpha m_4^{-\beta} k}
                 \frac{1}{2\pi i}\int_{a+i}^{a+iT} \chi(s-\beta)\bigg(\frac{m_1m_2m_3m_4h}{k}\bigg)^{-s}ds
      \\  = &  \sum_{m_1,m_2,m_3,m_4,h,k}\frac{\mu(m_1)a(h)\ol{a(k)}}{m_1^\beta m_2^{\beta+\gamma}m_3^{\alpha+\beta}h^\beta k^{1-\beta}}
                 \frac{1}{2\pi i}\int_{a-\beta+i}^{a-\beta+iT} \chi(s)\bigg(\frac{m_1m_2m_3m_4h}{k}\bigg)^{-s}ds
      \\  = &  \sum_{\substack{m\\ k\leq y}}\frac{b(m)\ol{a(k)}}{m^\beta k^{1-\beta}}
                 \frac{1}{2\pi i}\int_{a-\beta+i}^{a-\beta+iT} \chi(s)\bigg(\frac{m}{k}\bigg)^{-s}ds.
\end{split}
\ee
where 
\ben\label{Def b}
b(m)=\sum_{\substack{m_1m_2m_3m_4h=m\\h\leqs y}}\mu(m_1)m_2^{-\gamma}m_3^{-\alpha}m_4^\beta a(h).
\een
Note that we have 
\begin{align*}
b(m)\ll \tau_4* a(m)\ll \tau_{r+4}(m)(\log m)^C.
\end{align*}
This integral can be evaluated by the following lemma.
\begin{lem}\label{Integral I}
	Suppose that $B(s)=\sum_{n}b(n)n^{-s}$ for $\sigma>1$ where $a(n)\ll \tau_{k_1}(n)(\log n)^{l_1}$ for some non-negative integers $k_1$ and $l_1$. Let $A(s)=\sum_{n\leq y}a(n)n^{-s}$, where $a(n)\ll \tau_{k_2}(\log n)^{l_2}$ for some non-negative integers $k_2$,  $l_2$ and $T^\epsilon\ll y \leq T$ for some $\epsilon>0$. Then 
	\begin{align*}
	&\frac{1}{2\pi i}\int_{c+i}^{c+iT}\chi(1-s)A(1-s)B(s)ds\\&=\sum_{k\leq y}\frac{a(k)}{k}\sum_{m\leq nT/2\pi}b(m)e(-m/k)+O(yT^{1/2}(\log T)^{k_1+k_2+l_1+l_2}),
	\end{align*} 
where $c=1+1/\log T$.
\end{lem}
\begin{proof}
	 See \cite[Lemma 2]{CGGSimple}.
\end{proof}
Applying Lemma \ref{Integral I} we get that 
\be
\begin{split}
\mc{I} = & \sum_{k\leqs y}\frac{\ol{a(k)}}{k^{1-\beta}}\sum_{m\leqs Tk/2\pi}\frac{b(m)e(-m/k)}{m^\beta}
                +O(yT^{1/2+\epsilon}).
\end{split}
\ee
Following \cite{CGGSimple}, we now express the additive character $e(-m/k)$ in terms of multiplicative characters. We write $m'=m/(m,k)$ and $k'=k/(m,k)$ and then
\begin{align*}
e(-m/k)=e(-m'/k')&=\frac{1}{\phi(k')}\sum_{\chi \pmod {k'}}\tau(\bar{\chi})\chi(-m'	)\\&=\frac{\mu(k')}{\phi(k')}+ \frac{1}{\phi(k')}\sum_{\chi \not=\chi_0\pmod {k'}}\tau(\bar{\chi})\chi(-m')
\end{align*} 
where $\tau(\chi)$ denotes the Gauss sum. 
The first term here will lead to the main term whilst the second term will give rise to the error. When computing the error term we wish to apply the large sieve and hence it is necessary to express the sum over characters in terms of primitive characters.
To this end we write 
\begin{align*}
\frac{1}{\phi(k')}\sum_{\chi \not =\chi_0\pmod {k'}}\tau(\bar{\chi})\chi(-m')
=
\frac{1}{\phi(k')}\sum_{q\mid k', q>1\,\,} \sideset{}{^*}\sum_{\psi \pmod q}\mu\left( \frac{k'}{q}\right)\bar{\psi}\left(\frac{k'}{q}\right)\tau(\bar{\psi})\psi(-m')
\end{align*}
where the $*$ denotes that the sum is over primitive characters. After an application of M\"obius inversion as in \cite{CGGSimple}, formula (5.10), we have 
\begin{align*}
&\frac{1}{\phi(k')}\sum_{\chi \not =\chi_0\pmod {k'}}\tau(\bar{\chi})\chi(-m')\\&=\sum_{d\mid (m,k)}\sum_{e\mid d}\frac{\mu(d/e)}{\phi(k/e)}\sum_{q\mid k/e, q>1\,\,}
\sideset{}{^*}\sum_{\psi \pmod q}\mu\left( \frac{k}{eq}\right)\bar{\psi}\left( \frac{k}{eq}\right)\tau(\bar{\psi})\psi\left( -\frac{m}{e}\right)\\
&=\sum_{q\mid k , q>1\,\,} \sideset{}{^*}\sum_{\psi \pmod q}\tau(\bar{\psi})\sum_{d\mid (m,k)}\psi\left( \frac{m}{d}\right)\delta(q, k, d, \psi),
\end{align*}
where 
\begin{align*}
\delta(q, k, d, \psi)=\sum_{\substack{e\mid d\\ e\mid k/q}}\frac{\mu(d/e)}{\phi(k/e)}\bar{\psi}\left(- \frac{k}{eq}\right)\psi\left( \frac{d}{e}\right)\mu\left( \frac{k}{eq}\right).
\end{align*}
Therefore, we can write 
\ben\label{I decomp}
\mc{I}=\mc{M}+\mc{E}+O(yT^{1/2+\epsilon}),
\een
where 
\begin{align*}
&\mc{M}=\sum_{k\leq y}\frac{\ol{a(k)}}{k^{1-\beta}}\sum_{m\leq Tk/2\pi}\frac{b(m)}{m^\beta}\frac{\mu(k/(m,k))}{\phi(k/(m,k))},\\
&\mc{E}=\sum_{k\leq y}\frac{\ol{a(k)}}{k^{1-\beta}}\sum_{m\leq Tk/2\pi}\frac{b(m)}{m^\beta}\sum_{\substack{q\mid k\\ q>1}}
\sideset{}{^*}\sum_{\psi \pmod q}\tau(\bar{\psi})\sum_{d\mid (m,k)}\psi\left( \frac{m}{d}\right)\delta(q, k, d, \psi).
\end{align*}


\subsection{Computing the main term $\mc{M}$} 
We compute $\mathcal{M}$ essentially by applying Perron's formula to the inner sum although there are some arithmetic complications to deal with. We first unfold the definition of $b(m)$ 
\[
	\mc{M} =  \sum_{h,k\leqs y}\frac{a(h)\ol{a(k)}}{h^\beta k^{1-\beta}}\sum_{n\leqs Tk/2\pi h}\frac{c(n)}{n^\beta}\frac{\mu(k/(nh,k))}{\phi(k/(nh,k))},  
\]
where 
\ben
\label{c(n)}
c(n)=\sum_{\substack{n_1n_2n_3n_4=n}}\mu(n_1)n_2^{-\gamma}n_3^{-\alpha}n_4^\beta.
\een
We then group the terms $h,k$ according to their greatest common divisor $g=(h,k)$ and acquire the formula

\begin{align*}
\mathcal{M} = &\sum_{g\leq y} 
\sum_{\substack{h,k\leq y/g\\ (h,k)=1}} \frac{a(gh)\ol{a(gk)}}{gh^\beta k^{1-\beta}}
\sum_{n\leq Tk/2\pi h} \frac{c(n)}{n^\beta}\frac{\mu(k/(n,k))}{\phi(k/(n,k))}.
\end{align*}
On grouping together terms for which $(n,k)=d$ we get  
\begin{align}   
	 \label{final M}      
\mc{M} 
= \sum_{g\leq y} 
\sum_{\substack{h\leq y/g}} \frac{a(gh)}{gh^\beta }
\sum_{d\leq y/g}
\sum_{\substack{k\leq y/dg\\ (h,dk)=1}}\frac{\ol{a(gdk)}}{dk^{1-\beta}}\frac{\mu(k)}{\phi(k)}
\sum_{\substack{n\leq Tk/2\pi h\\(n,k)=1}}\frac{c(dn)}{n^\beta}.
\end{align}
To encode the dependence of $d$ in the innermost sum, we use the following lemma.
\begin{lem}\label{DecompositionArithmeticFunction}
	Let $j, D\in \mathbb{N}$ and let $f_1, \dots, f_j$ be arithmetic functions. Given a decomposition of integers $D=d_1\cdots d_j$, define $D_i=\prod_{u=1}^{j-i}d_u$ for $1\leq i \leq j-1$ and $D_j=1$. The following identities hold:
	\begin{align*}
	&\sum_{\substack{m\leq x\\ (m,k)=1}}(f_1*f_2* \cdots *f_j)(mD)
	=\sum_{d_1 \cdots d_j=D}\sum_{\substack{m_1\cdots m_j\leq x\\ (m_i,kD_i)=1}}f_1(m_1d_j)f_2(m_2d_{j-1})\cdots f(m_jd_1),\\
	&\sum_{(m,k)=1}\frac{(f_1* f_2* \cdots *f_j)(mD)}{m^s}=\sum_{d_1 \cdots d_j=D}\prod_{i=1}^j\sum_{(m_i, kD_i)=1}\frac{f(m_id_{j+1-i})}{m_i^s}.
	\end{align*}
\end{lem}
\begin{proof}The second identity is Lemma 3 of \cite{CGGSimple}. The first identity follows from the same method of proof.
\end{proof}
From Lemma \ref{DecompositionArithmeticFunction} we see that the corresponding Dirichlet series may be written as
\begin{equation}
\begin{split}
\label{dirichlet series}
&\sum_{(n,k)=1}  \frac{c(nd)}{n^{s+\beta}}\\
& = \sum_{d_1d_2d_3d_4=d}\sum_{(m_1,kd_1d_2d_3)=1}\mu(m_1d_4)\sum_{ (m_2,kd_1d_2)=1} (m_2d_3)^{-\gamma}\sum_{ (m_3, kd_1)=1} (m_3d_2)^{-\alpha}
\\ 
& \qquad \times \sum_{ (m_4, k)=1} (m_4d_1)^\beta(m_1m_2m_3m_4)^{-(s+\beta)}
\\
&  =  \sum_{d_1d_2d_3d_4=d} \mu(d_4)d_3^{-\gamma}d_2^{-\alpha}d_1^\beta \sum_{(m_1, kd)=1}\mu(m_1) m_1^{-s-\beta} \sum_{(m_2,kd_1d_2)=1}m_2^{-\gamma-s-\beta}
\\  
&   \qquad \times  \sum_{ (m_3, kd_1)=1} m_3^{-\alpha-s-\beta}\sum_{(m_4, k)=1} m_4^{-s}
\\
& = \frac{\zeta(s)\zeta(s+\gamma+\beta)\zeta(s+\alpha+\beta)}{\zeta(s+\beta)}G(s,k,d)
\end{split}
\end{equation}
where 
\begin{multline}\label{G}
    G(s,k,d)  =  \sum_{d_1d_2d_3d_4=d} \mu(d_4)d_3^{-\gamma}d_2^{-\alpha}d_1^\beta
                       \prod_{p\mid kd} (1-p^{-s-\beta})^{-1}
        \\  \times  \prod_{p\mid kd_1d_2}(1-p^{-s-\gamma-\beta})
                       \prod_{p\mid k d_1} (1-p^{-\alpha-s-\beta})
                       \prod_{p\mid k}(1-p^{-s}).
\end{multline}
Note that $G(s,k,d)$ is holomorphic in the region $\sigma>0$ and that for $\sigma\geqs 1/2$, say, we have the bound 
\ben
\label{G3 bound}
G(s,k,d)\ll \tau_4(d)\prod_{p|kd}\Big(1+10p^{-\sigma}\Big)\ll \tau_4(d)\tau(kd)
\een
since the shifts are all $\ll 1/\log T$ and $d,k\ll T$. 
Also, note by changing the role of $d_2,d_3, d_4$, we can write 
\begin{multline}\label{G2}
G(s,k,d)  =  \sum_{d_1d_2d_3d_4=d} \mu(d_4)d_3^{-\gamma}d_2^{\beta}d_1^{-\alpha}
\prod_{p\mid kd} (1-p^{-s-\beta})^{-1}
\\  \times  \prod_{p\mid kd_1d_2}(1-p^{-s-\gamma-\beta})
\prod_{p\mid k d_1} (1-p^{-s})
\prod_{p\mid k}(1-p^{-s-\alpha-\beta}).
\end{multline}
as well as 
\begin{multline}\label{G3}
G(s,k,d)  =  \sum_{d_1d_2d_3d_4=d} \mu(d_4)d_3^{\beta}d_2^{-\alpha}d_1^{-\gamma}
\prod_{p\mid kd} (1-p^{-s-\beta})^{-1}
\\  \times  \prod_{p\mid kd_1d_2}(1-p^{-s})
\prod_{p\mid k d_1} (1-p^{-\alpha-s-\beta})
\prod_{p\mid k}(1-p^{-s-\gamma-\beta}).
\end{multline}
These alternative formulations will be useful when recovering the second and third $Z$ terms of \eqref{mc I}.

\subsubsection{Perron's formula} We employ the following version of Perron's formula to evaluate the innermost sum in $\mathcal{M}$.

\begin{lem}\label{Perron}[Theorem 2.1 \cite{LiuYe}]
	Let $f(s)=\sum_{n=1}^\infty a_nn^{-s}$ be a Dirichlet series with abscissa of absolute convergence $\sigma_a$. Let  \begin{align}
	B(\sigma)=\sum_{n=1}^\infty \frac{|a_n|}{n^\sigma},
	\end{align}
	for $\sigma>\sigma_a$. Then for $\kappa> \sigma_a$, $x\geq 2$, $U\geq 2$, and $H\geq 2$, we have 
	\begin{align}
	\sum_{n\leq x} a_n= \frac{1}{2\pi i}\int_{\kappa-iU}^{\kappa+iU} f(s)\frac{x^s}{s}ds+ O \bigg( \sum_{ x-x/H \leq n \leq x+x/H}|a_n|\bigg)+ O \left( \frac{x^\kappa H B(\kappa)}{U}\right).
	\end{align}
\end{lem}

Applying Lemma \ref{Perron} with $U=\exp (c\sqrt{\log T}), H=\sqrt{U}$ and $\kappa=1+1/\log T$ we find that 
\be
\begin{split}
	\sum_{\substack{n\leqs Tk/2\pi h\\(n,k)=1}}\frac{c(dn)}{n^\beta}
	= &  \frac{1}{2\pi i}\int_{\kappa-iU}^{\kappa+iU}
	\frac{\zeta(s)\zeta(s+\alpha+\beta)\zeta(s+\beta+\gamma)}{\zeta(s+\beta)}G(s,k,d)
	\bigg(\frac{Tk}{2\pi h}\bigg)^s\frac{ds}{s} \\
	+ & O\bigg(\sum_{ \frac{kT}{2\pi h}-\frac{kT}{2\pi h\sqrt{U}} \leq n \leq \frac{kT}{2\pi h}+ \frac{kT}{2\pi h \sqrt{U}}}\left|\frac{c(dn)}{n^\beta}\right|\bigg)   
	+    O\bigg(\frac{Tk}{h}\frac{(\log T)^C}{\sqrt{U}}\bigg)
\end{split}
\ee
From \eqref{c(n)}, we have $|c(n)|\ll \tau_4(n)$ for $n\leqs T^{O(1)}$, and thus by Shiu's bound for short divisor sums \cite[Theorem 2]{S} we have
\begin{align*}
\sum_{ \frac{kT}{2\pi h}-\frac{kT}{2\pi h\sqrt{U}} \leq n \leq \frac{kT}{2\pi h}+ \frac{kT}{2\pi h \sqrt{U}}}\left|\frac{c(dn)}{n^\beta}\right| 
\ll &
 \sum_{ \frac{kT}{2\pi h}-\frac{kT}{2\pi h\sqrt{U}} \leq n \leq \frac{kT}{2\pi h}+ \frac{kT}{2\pi h \sqrt{U}}} \tau_4(d)\tau_4(n)
\\
\ll & \tau_4(d)\frac{kT}{h\sqrt{U}}(\log T)^4.
\end{align*}
Therefore the error terms contribute to $\mathcal{M}$ at most 
\begin{align*}
&\sum_{g\leq y}\sum_{h\leq y/g}\frac{|a(gh)|}{gh^\beta}\sum_{d\leq y/g}\sum_{k\leq y/dg}\frac{|a(gdk)|}{dk^{1-\beta}}\frac{\tau_4(d)}{\phi(k)} \frac{Tk}{h\sqrt{U}}(\log T)^C\\
& \ll  \sum_{g\leq y}\frac{|a(g)|^2}{g}\sum_{h\leq y}\frac{|a(h)|}{h} \sum_{d\leq y/g}\frac{|a(d)|\tau_4(d)}{d}\sum_{k\leq y/dg}\frac{|a(k)|}{k}\frac{T}{\sqrt{U}}(\log T)^{C'}\\
& \ll \frac{T}{\sqrt{U}}(\log T)^{c''}
\end{align*}
where we have used that $a(n)\ll \tau_r(n)(\log n)^C$ and $a(mn)\ll |a(m)a(n)|$.

Moving the contour to the line $\sigma=1-\frac{c}{\log U}$, we encounter three poles. By \eqref{G3 bound} and since $\zeta(s)^{\pm 1}\ll \log U$ in the zero free region, the integral over the left edge of the contour leads to a contribution
\begin{multline*}
(\log U)^4\sum_{g\leq y}\sum_{h\leq y/g}\frac{|a(gh)|}{gh^\beta}\sum_{d\leq y/g}\sum_{k\leq y/dg}\frac{|a(gdk)|}{dk^{1-\beta}} \frac{\tau_4(d)\tau(kd)}{\phi(k)} \Big(\frac{Tk}{h}\Big)^{1-\frac{c}{\log U}}
\\
\ll 
T\exp (-c\sqrt{\log T}).
\end{multline*}
The integral over the horizontal lines contribute at most 
\begin{multline*}
\frac{(\log U)^4}{U}\sum_{g\leq y}\sum_{h\leq y/g}\frac{|a(gh)|}{gh^\beta}\sum_{d\leq y/g}\sum_{k\leq y/dg}\frac{|a(gdk)|}{dk^{1-\beta}}\frac{\tau_4(d)\tau(kd)}{\phi(k)}\Big(\frac{Tk}{h}\Big)
\\
\ll T \exp (-c\sqrt{\log T}).
\end{multline*}
 Therefore, we arrive at the following:
\begin{multline}
\label{res comp}
        \mc{M}  =  \sum_{g\leq y} 
               \sum_{\substack{h\leq y/g}} \frac{a(gh)}{gh^\beta }
               \sum_{d\leq y/g}
               \sum_{\substack{k\leq y/dg\\ (h,dk)=1}}\frac{\ol{a(gdk)}}{dk^{1-\beta}}\frac{\mu(k)}{\phi(k)}
               \sum_{\substack{z=1\\z=1-\alpha-\beta\\z=1-\beta-\gamma}}\text{res}_{s=z}(\mc{F}(s))
\\               +O(T\exp(-c\sqrt{\log T})),
\end{multline}  
where
\[
\mc{F}(s)=\frac{\zeta(s)\zeta(s+\alpha+\beta)\zeta(s+\beta+\gamma)}{\zeta(s+\beta)}G(s,k,d)
         \frac{1}{s} \bigg(\frac{Tk}{2\pi h}\bigg)^s.
\]


\subsubsection{Computing the residues} 
Let us first analyse the contribution from the residue at $s=1$. This is given by
\begin{multline*}
   \frac{T}{2\pi}
   \sum_{g\leqs y}\sum_{\substack{h\leqs y/g}}\frac{a(gh)}{gh^{1+\beta}}\sum_{d\leqs y/g}
   \sum_{\substack{k\leqs y/dg\\(h,dk)=1}}\frac{\ol{a(gdk)}}{dk^{-\beta}}\frac{\mu(k)}{\phi(k)}
   \frac{\zeta(1+\alpha+\beta)\zeta(1+\beta+\gamma)}{\zeta(1+\beta)}G(1,k,d) 
   \\
 =    \sum_{g\leqs y}\sum_{\substack{h,k^\prime\leqs y/g\\(h,k^\prime)=1}}\frac{a(gh)\ol{a(gk^\prime)}}{ghk^\prime}
          \int_0^{T/2\pi} \frac{1}{h^\beta}  \frac{\zeta(1+\alpha+\beta)\zeta(1+\beta+\gamma)}{\zeta(1+\beta)}  \\
     \times     \sum_{kd=k^\prime} k^{\beta}\frac{k}{\phi(k)}\mu(k) G(1,k,d) dt.
\end{multline*}
We will show that the integrand is given by $Z_{\alpha,\beta,\gamma,h,k}$. From the definition of $Z$ we are required to show that 
\ben
\label{identity}
\sum_{kd=k^\prime} k^{\beta}\frac{k}{\phi(k)}\mu(k) G(1,k,d)
=
\prod_{p^{{k_p^\prime}} || k^\prime}\frac{\sum_{m\geqs 0}f_{\alpha,\gamma}(p^{m+{k_p^\prime}})p^{-m(1+\beta)}}{\sum_{m\geqs 0}f_{\alpha,\gamma}(p^{m})p^{-m(1+\beta)}}
\een
where we recall that 
\begin{align*}
f_{\alpha,\gamma}(n) = & \sum_{n_1n_2n_3=n}\mu(n_1)n_2^{-\alpha}n_3^{-\gamma}.
\end{align*}

To prove this identity we first manipulate the right hand side of \eqref{identity}. Note that 
\begin{align*}
f_{\alpha,\gamma}(p^{m})=\frac{p^{-m\gamma}-p^{-\alpha(m+1)+\gamma}-p^{-(m-1)\gamma}+p^{-\alpha m+\gamma}}{1-p^{\gamma-\alpha}}
\end{align*}
and hence
\begin{align*}
\prod_{p\mid k^\prime} \sum_{m\geq 0}f_{\alpha, \gamma}(p^{m})p^{-m(1+\beta)}=\prod_{p\mid k^\prime} \frac{(1-p^{-(1+\beta)})}{(1-p^{-(1+\alpha+\beta)})(1-p^{-(1+\gamma+\beta)})}
\end{align*}
and 
\begin{align*}
&\prod_{p^{k^\prime_p}\| k^\prime} \sum_{m\geq 0} f_{\alpha, \gamma}(p^{m+k^\prime_p})p^{-m(1+\beta)}\\
&=\prod_{p^{k^\prime_p}\| k^\prime} \frac{p^{\alpha +\beta +\gamma +1} \left(\frac{\left(p^{\gamma }-1\right) p^{-\gamma {k^\prime_p}}}{p^{\beta +\gamma+1}-1}-\frac{\left(p^{\alpha }-1\right) p^{-\alpha  {k^\prime_p}}}{p^{\alpha +\beta +1}-1}\right)}{p^{\gamma }-p^{\alpha }}\\
&=\prod_{p^{k^\prime_p}\| k^\prime}\bigg(\frac{p^{-\gamma k^\prime_p}(1-p^{-\gamma})}{(1-p^{-(1+\gamma+\beta)})(p^{-\alpha}-p^{-\gamma})}-\frac{p^{-\alpha k^\prime_p}(1-p^{-\alpha})}{(1-p^{-(1+\alpha+\beta)})(p^{-\alpha}-p^{-\gamma})}\bigg).
\end{align*}
Therefore, the right hand side of \eqref{identity} is given by
\begin{align}
\label{rhs}
\prod_{p\mid k^\prime} \bigg(p^{-\alpha k^\prime_p}\frac{(1-p^{-(1+\gamma+\beta)})(1-p^{-\alpha})}{(1-p^{-(1+\beta)})(p^{-\gamma}-p^{-\alpha})}+p^{-\gamma k^\prime_p} \frac{(1-p^{-(1+\alpha+\beta)})(1-p^{-\gamma})}{(1-p^{-(1+\beta)})(p^{-\alpha}-p^{-\gamma})}\bigg).
\end{align}

For the left hand side, inputting the definition of $G$ given in \eqref{G} we find 
\begin{align*}
\sum_{kd=k^\prime} k^{\beta}\frac{k}{\phi(k)}\mu(k) G(1,k,d)
= &
\prod_{p\mid k^\prime} (1-p^{-1-\beta})^{-1}
\sum_{kd=k^\prime} \mu(k)
\sum_{d_1d_2d_3d_4=d} \mu(d_4)d_3^{-\gamma}d_2^{-\alpha}(kd_1)^\beta
\\  
 & \times  
\prod_{p\mid kd_1d_2}(1-p^{-1-\gamma-\beta})
\prod_{p\mid k d_1} (1-p^{-1-\alpha-\beta}).                    
\end{align*}
We first combine the two sums and write them as a single sum over the condition $kd_1d_2d_3d_4=k^\prime$. We then write $kd_1$ as $\ell$ and acquire
 \begin{align*}
\prod_{p\mid k^\prime} (1-p^{-1-\beta})^{-1}
\sum_{\ell d_2d_3d_4=k^\prime} 
\mu(d_4)d_3^{-\gamma}d_2^{-\alpha}\ell^\beta
 \prod_{p\mid \ell d_2}(1-p^{-1-\gamma-\beta})
 \prod_{p\mid \ell} (1-p^{-1-\alpha-\beta})  
 \sum_{kd_1=\ell}\mu(k)
 \\
 =
 \prod_{p\mid k^\prime} (1-p^{-1-\beta})^{-1}
\sum_{d_2d_3d_4=k^\prime} 
\mu(d_4)d_3^{-\gamma}d_2^{-\alpha}
\prod_{p\mid  d_2}(1-p^{-1-\gamma-\beta}).
 \end{align*}
For the sum over $d_i$ we have 
\begin{align*}
\sum_{d_2d_3d_4=k^\prime}&\mu(d_4)d_3^{-\gamma}d_2^{-\alpha} \prod_{p\mid d_2}(1-p^{-(1+\gamma+\beta)})\\
&=\prod_{ p^{k^\prime_p}\| k^\prime} \bigg(\sum_{1\leq m\leq k^\prime_p} p^{-\gamma(k^\prime_p-m)}p^{-\alpha m}(1-p^{-(1+\gamma+\beta)})\\
&\ \ \ \ \ \ \ \ \ \ \ \ -\sum_{1\leq m\leq k^\prime_p-1} p^{-\gamma(k^\prime_p-1-m)}p^{-\alpha m}(1-p^{-(1+\gamma+\beta)})\\
&\ \ \ \ \ \ \ \ \ \ \ \ +p^{-\gamma k^\prime_p}-p^{-\gamma(k^\prime_p-1)}\bigg)\\
&=\prod_{p^{k^\prime_p}\| k^\prime} \bigg(p^{-\alpha k^\prime_p}\frac{(1-p^{-(1+\gamma+\beta)})(1-p^{-\alpha})}{(p^{-\gamma}-p^{-\alpha})}+p^{-\gamma k^\prime_p} \frac{(1-p^{-(1+\alpha+\beta)})(1-p^{-\gamma})}{(p^{-\alpha}-p^{-\gamma})}\bigg)
\end{align*}
and thus after multiplying by $\prod_{p\mid k^\prime} (1-p^{-1-\beta})^{-1}$ this is equal to \eqref{rhs}. Equation \eqref{identity} then follows.

When computing the residue at $s=1-\alpha-\beta$ in \eqref{res comp}, we get a factor of 
\[
\frac{1}{1-\alpha-\beta}\bigg(\frac{T}{2\pi}\bigg)^{1-\alpha-\beta}=\int_0^{T/2\pi} \bigg(\frac{t}{2\pi}\bigg)^{-\alpha-\beta}dt.
\]
In the arithmetic sums we see that the effect of changing $s$ from 1 to $1-\alpha-\beta$ is to interchange $\alpha$ with $-\beta$ and $\beta$ with $-\alpha$ after using the expression  for $G(s, k, d)$ given in \eqref{G2}. This is precisely the behaviour of the second $Z$ term in \eqref{twisted moment eq} and hence we acquire this term. Likewise, for the residue at $s=1-\beta-\gamma$, the effect of changing $s$ from 1 to $1-\beta-\gamma$ is to interchange $\gamma$ with $-\beta$ and $\beta$ with $-\gamma$ after using the expression \eqref{G3} for $G(s, k, d)$. This gives the third and final $Z$ term.


\subsection{Bounding the error term $\mathcal{E}$}
In this section we give unconditional bounds for $\mathcal{E}$ when $a(n)$ satisfies the following properties 
\begin{align}
a(mn)\ll |a(m)a(n)|, \ \ a(n)\ll \tau_r(n)(\log n)^C.
\end{align}
The error term $\mc{E}$ in \eqref{I decomp} can be rewritten as 
\begin{align*}
\mathcal{E}
&=\sum_{2\leq q\leq y} \sideset{}{^*}\sum_{\psi\pmod q}\tau(\bar{\psi}) \sum_{k\leq y/q}\frac{a(kq)}{(kq)^{1-\beta}}\sum_{d\mid kq}\delta(q,kq,d,\psi)\sum_{m\leq Tkq/2\pi d}\frac{b(md)\psi(m)}{(md)^\beta}\\
&=:\mathcal{E}_1+\mathcal{E}_2
\end{align*}
where 
\begin{align}
&\mathcal{E}_1= \sum_{2\leq q\leq \eta} \sideset{}{^*}\sum_{\,\,\psi \pmod q} \tau(\bar{\psi}) \sum_{kq\leq y} \frac{a(kq)}{(kq)^{1-\beta}}\sum_{d\mid kq} \delta(q, kq, d, \psi)\sum_{m\leq Tkq/2\pi d}\frac{b(md)\psi(m)}{(md)^\beta}\\
&\mathcal{E}_2=\sum_{\eta \leq q\leq y} \sideset{}{^*}\sum_{\,\,\psi \pmod q} \tau(\bar{\psi}) \sum_{kq\leq y} \frac{a(kq)}{(kq)^{1-\beta}}\sum_{d\mid kq} \delta(q, kq, d, \psi)\sum_{m\leq Tkq/2\pi d}\frac{b(md)\psi(m)}{(md)^\beta}
\end{align}
We use Siegel's theorem to bound $\mathcal{E}_1$ and the large sieve inequalities to bound $\mathcal{E}_2$.

\begin{prop}[Small moduli]\label{E1}
If $a(nm)\ll |a(m)a(n)|$ and $a(n)\ll \tau_r(n)(\log n)^C$.	Then for $\eta\ll (\log T)^{A}$, we have
	\begin{align}
	\mathcal{E}_1\ll T\exp(-c\sqrt{\log T})\eta^{3/2+\epsilon}
	\end{align}
\end{prop}

\begin{prop}[Large moduli]\label{E2}
	If $a(nm)\ll |a(m)a(n)|$ and $a(n)\ll \tau_r(n)(\log n)^C$, then for some $C'=C'(r, A)$, we have
	\begin{align}
	\mathcal{E}_2\ll (\log T)^{C'} T\eta^{-1/2+\epsilon}+ yT^{1/2+\epsilon}+ y^{4/3}T^{1/3+\epsilon}+y^{1/3+\epsilon}T^{5/6+\epsilon}.
	\end{align}
\end{prop}
\begin{proof}[Proof of Theorem \ref{twisted moment thm}]
	Combining Proposition \ref{E1} and Proposition \ref{E2}, with $\eta=(\log T)^{C'}$ for $C'$ large enough, we find that $\mc{E}\ll T(\log T)^{-A}$ provided $y=T^{\theta}$ for some fixed $\theta<1/2$. Theorem \ref{twisted moment thm} then follows.
\end{proof}


\section{Proof of Proposition \ref{E1}}\label{SmallmoduliProof}
To prove Proposition \ref{E1}, we use the following lemma.
\begin{lem}\label{NontirivalCharacter}
	Let $\psi \pmod q$ be a non-principal character with $q\ll (\log T)^A $. Then for $T\ll x\ll T^2$ and $d\ll T$, we have
	\begin{align*}
	\sum_{m\leq x} \frac{b(md)\psi(m)}{m^\beta}\ll x\exp (-c\sqrt{\log x}) (\tau_4*|a|)(d)j(d),
	\end{align*}
	where 
	\begin{align}
	j(d)=\prod_{p\mid d} (1+10p^{-1/2}).
	\end{align}
\end{lem}
\begin{proof}
After an application of Lemma \ref{Perron} with $\kappa=1+O(1/\log x)$, $H=\sqrt{U}$ with $U$ to be determined later, we have 
\begin{align*}
\sum_{m\leq x}\frac{b(md)\psi(m)}{m^\beta}=\frac{1}{2\pi i}\int_{\kappa-iU}^{\kappa+iU}\sum_{m}\frac{b(md)\psi(m)}{m^{\beta+s}}x^s \frac{ds}{s}+E
\end{align*}
where 
\begin{align*}
E\ll \sum_{x-x/\sqrt{U}\leqs m\leqs x+x/\sqrt{U}}|b(md)\psi(m)| +\frac{x}{\sqrt{U}}\sum_{n\geqs 1}|b(md)\psi(m)|m^{-\kappa}.
\end{align*}
From \eqref{Def b}, we have 
\begin{align*}
b(md)\ll (md)^\beta (\tau_4* |a|)(md).
\end{align*}
Therefore, the second term above is 
$\ll ({x}/{\sqrt{U}})(\tau_4* |a|)(d)(\log x)^C$ and in fact, the same bound holds for the first term. To see this we use Lemma 6.4 of \cite{Ngdiscrete} which states that  
\begin{align*}
\sum_{t-u\leq n\leq t}(\tau_k*a)(n)\ll u (\log t)^{k-1}\|a(n)/n\|_1
\end{align*}
for $x/2\leq t-u \leq t\leq x$, $T\ll x\leq T^2$, $u=x/U$ with $\exp(c\sqrt{\log x})\leq U \leq \left( \frac{\log x}{\log \log x}\right)$, and $a(n)$ supported on integers $\leq y\leq \sqrt{T}$.
Therefore, 
\begin{align}
E\ll \frac{x}{\sqrt{U}}(\tau_4*|a|)(d)(\log x)^C.
\end{align}
Now it remains to compute 
\begin{align*}
\int_{\kappa-iU}^{\kappa+iU}\sum_{m}\frac{b(md)\psi(m)}{m^{\beta+s}}x^s \frac{ds}{s}
\end{align*}
We shall move the contour to the line $\Re(s)=1-c/\log (qU)$ for some absolute $c$. To do this we first express the
function $\sum_{m}b(md)\psi(m)m^{-\beta-s}$ in terms of $L$-functions.  
 
Applying Lemma \ref{DecompositionArithmeticFunction}, we have 
\begin{align*}
&\sum_{m}\frac{b(md)\psi(m)}{m^{s+\beta}}\\&=\sum_{d_1d_2d_3d_4d_5=d}\sum_{(m_1, d_1d_2d_3d_4)=1}\frac{\mu(m_1d_5)\psi(m_1)}{m_1^{s+\beta}}\sum_{(m_2,d_1d_2d_3)=1}\frac{(m_2d_4)^{-\gamma}\psi(m_2)}{m_2^{s+\beta}}
\\& \times \sum_{(m_3,d_1d_2)=1}\frac{(m_3d_3)^{-\alpha}\psi(m_3)}{m_3^{s+\beta}}\sum_{(m_4,d_1)=1}\frac{(m_4d_2)^\beta\psi(m_4)}{m_4^{s+\beta}}\sum_{m_5}\frac{a(m_5d_1)\psi(m_5)}{m_5^{s+\beta}}\\
& =\sum_{\prod_{i=1}^5d_i=d}\frac{\mu(d_5)}{L(s+\beta, \psi)}\prod_{p\mid d}\left( 1-\frac{\psi(p)}{p^{s+\beta}}\right)^{-1}d_4^{-\gamma}L(s+\beta+\gamma, \psi)\prod_{p\mid d_1d_2d_3}\left(1-\frac{\psi(p)}{p^{s+\beta+\gamma}}\right)\\
& \times d_3^{-\alpha} L(s+\beta+\alpha, \psi)\prod_{p\mid d_1d_2}\left( 1-\frac{\psi(p)}{p^{s+\beta+\alpha}}\right)d_2^\beta L(s, \psi)\prod_{p\mid d_1}\left(1-\frac{\psi(p)}{p^s}\right)\sum_{m_5}\frac{a(m_5d_1)\psi(m_5)}{m_5^{s+\beta}}\\
&=\frac{L(s+\beta+\gamma, \psi)L(s+\beta+\alpha, \psi)L(s, \psi)}{L(s+\beta, \psi)}\prod_{p\mid d}\left( 1-\frac{\psi(p)}{p^{s+\beta}}\right)^{-1}\\& \times \sum_{\prod_{i=1}^5 d_i=d}\mu(d_5)d_4^{-\gamma}d_3^{-\alpha}d_2^\beta\prod_{p\mid d_1d_2d_3}\left(1-\frac{\psi(p)}{p^{s+\beta+\gamma}}\right)\prod_{p\mid d_1d_2}\left(1-\frac{\psi(p)}{p^{s+\beta+\alpha}}\right)\prod_{p\mid d_1}\left(1-\frac{\psi(p)}{p^s}\right)A(s+\beta, d_1),
\end{align*}
where 
\begin{align*}
A(s, r)=\sum_{m}\frac{a(mr)\psi(m)}{m^s}=\sum_{mr\leq y} \frac{a(mr)\psi(m)}{m^s}.
\end{align*}

To bound the horizontal integrals when moving the contour, we need some bounds for the $\sum_{m}b(md)\psi(m)m^{-\beta-s}$ for $\Re(s)=1-O(1/\log q\Im s)$ with $\Im s\gg 1$. 
Assuming that 
\begin{align*}
a(mn)\ll |a(m)||a(n)|,
\end{align*}
we have 
\begin{align*}
A(s,r)\ll |a(r)| \sum_{m\leq y}\frac{|a(m)|}{m^{\Re(s)}}\ll |a(r)| \|a(n)/n\|_1 y^{1-\Re(s)}\ll |a(r)| (\log x)^C y^{1-\Re(s)}.
\end{align*}
We also have for $1-\Re(s)\ll \frac{1}{\log q |\Im s|}$ and $ \Im s\gg1$
\begin{align*}
&\frac{1}{(\log q|\Im s| )^c}\ll L(s,\psi)\ll (\log q |\Im s|)^c,\\
\end{align*}
Therefore, when $1-\Re(s)\ll \frac{1}{\log q |\Im s|}$ and $\Im s\gg 1$, we have
\begin{align}\label{b(md)bound}
\sum_{m} \frac{b(md)\psi(m)}{m^{s+\beta}}\ll (\log q|\Im s|)^C j(d)(\tau_4*|a|)(d)(\log x)^Cy^{1-\Re(s)}.
\end{align}

There is at most one simple pole for $\sum_{m}b(md)\psi(m)m^{-s-\beta}$ for all non principal characters $\psi\pmod q$ with $q\ll T$ in the region $\{s=\sigma+it| \sigma\geq \sigma_1(t):= 1-c/\log q (|t|+2) \}$, where $c$ is some absolute constant. By Siegel's theorem, if this pole exists, then it is a real number $\beta$ such that $1-\beta \gg q^{-\epsilon}$. Thus,
\begin{align}
\sum_{m\leq x}\frac{b(md)\psi(s)}{m^\beta}&\ll \int_{\sigma_1(U)-iU}^{\sigma_1(U)+iU}\sum_{m} \frac{b(md)\psi(m)}{m^{s+\beta}}x^s\frac{ds}{s}+\left| \operatorname{Res}_{s=\beta}\sum_{m}\frac{b(md)\psi(m)}{m^{s+\beta}}\frac{x^s}{s}\right|\nonumber
\\
&+\frac{x}{{U}}(\tau_4*|a|)(d)(\log (qUx))^C +\frac{x}{\sqrt{U}}(\tau_4*|a|)(d)(\log x)^C,\label{horizontal}
\end{align}
where the third term is the contribution from the horizontal integrals using \eqref{b(md)bound}. 
Using \eqref{b(md)bound} again for the first integral we have
\begin{align}\label{leftIntegral}
\int_{\sigma_1(U)-iU}^{\sigma_1(U)+iU}\sum_{m}\frac{b(md)\psi(m)}{m^{s+\beta}}x^s \frac{ds}{s}
&\ll (\log (q U x))^C j(d)(\tau_4* |a|)(d) x^{\sigma_1(U)}\nonumber\\
&\ll j(d)(\tau_4*|a|)(d) x \exp (-c \frac{\log x}{\log q U}).
\end{align}
For the residue at $\beta$, we have 
\begin{align}\label{ResidueBeta}
x^\beta \ll x^{1-q^{-\epsilon}} \ll x \exp (-\frac{\log x}{q^\epsilon})\ll x\exp (- \frac{\log x}{(\log x)^{\epsilon A}})\ll x\exp (-c' \sqrt{\log x})
\end{align}
by choosing $\epsilon\ll \frac{1}{2A}$. 
Combining \eqref{horizontal}, \eqref{leftIntegral} \eqref{ResidueBeta} and choosing $U=\exp (c\sqrt{\log x})$, we have for $q\leq (\log x)^A$,
\begin{align}
\sum_{m} \frac{b(md)\psi(m)}{m^{\beta}}\ll_A j(d)(\tau_4*|a|)(d) x \exp(-c'\sqrt{\log x}).
\end{align} 
\end{proof}
\begin{proof}[Proof of Proposition \ref{E1}]
From Lemma 6.6 in \cite{Ngdiscrete}, we have 
\begin{align}\label{deltaBound}
|\delta(q, kq, d, \psi)|\ll \frac{(d, k)\log \log T}{\phi(k)\phi(q)}
\end{align}
for primitive characters $\psi$ and $kq\ll T$. 
From Lemma 6.7 in \cite{Ngdiscrete}, we also have 
\begin{align}\label{(d, k)sum}
\sum_{d\mid kq}\frac{(d,k)h(d)}{d}\ll (1*h)(k)\|h(n)/n\|_1
\end{align}
for positive multiplicative functions $h$. 
Using \eqref{deltaBound}, \eqref{(d, k)sum} and properties of $a(n)$, we have
\begin{align*}
\mathcal{E}_1&=\sum_{2\leq q\leq \eta} \sideset{}{^*}\sum_{\,\,\psi \pmod q} \tau(\bar{\psi}) \sum_{kq\leq y}\frac{\overline{a(kq)}}{(kq)^{1-\beta}} \sum_{d\mid kq}\delta(q, kq, d, \psi) \sum_{m\leq Tkq/2\pi d} \frac{b(md)\psi(m)}{(md)^\beta}\\
&\ll \sum_{q\leq \eta} \phi(q)\sqrt{q} \sum_{kq\leq y}\frac{|a(kq)|}{kq}\sum_{d\mid kq} \frac{(d, k)\log \log T}{\phi(k)\phi(q)} j(d)(\tau_4*|a|)(d)\frac{Tkq}{d} \exp(-c'\sqrt{\log T})\\
& \ll \sum_{q\leq \eta}q^{3/2}\sum_{kq\leq y}\frac{|a(kq)|}{kq} \sum_{d\mid kq}\frac{(d, k)j(d)(\tau_4* |a|)(d)}{d}T \exp (-c''\sqrt{\log T})\\
& \ll \sum_{q\leq \eta } |a(q)|q^{1/2} \sum_{kq\leq y}\frac{|a(k)|}{k}\sum_{d\mid kq} \frac{(d, k) \tau(d)(\tau_4* |a|)(d)}{d}T \exp(-c'''\sqrt{\log T}) \\
& \ll \sum_{q\leq \eta}|a(q)|\tau(q)q^{1/2}\sum_{k\leq y}\frac{|a(k)| \tau(k)(\tau_5*|a|)(k)}{k} \|(\tau_4*|a|)(n)/n\|_1T\exp(-c'''\sqrt{\log T})\\
& \ll \eta^{3/2+\epsilon} T \exp( -c''''\sqrt{\log T})
\end{align*}	
where we have used $j(d)\ll \tau(d)\leqs \tau(k)\tau(q)$.   
\end{proof}


\section{Proof of Proposition \ref{E2}}\label{LargemoduliPropProof}

\subsection{Initial cleaning}
The proof of Proposition \ref{E2} is similar to \cite{HBB}, and we give the exposition by considering Type I/II terms. One main difference is that our coefficients $a(n)$ are not supported on square-free integers. This affects the treatment of $\delta(q, kq, d, \psi)$ in the initial cleaning stage to get rid of the $q$ dependence on $d$ in the sum $d\mid kq$. In our case, $a(n)$ is not supported on square-free integers, but we can still remove the condition $d\mid q$ by exploiting the fact that $\psi $ has conductor $q$. 

	We write $k=k' k_q$, where $(k',q)=1$ and $k_q$ is such that $p|k_q\implies p|q$. Then we have 
	\begin{align*}
	\delta(q, kq, d, \psi)=\sum_{e\mid (d, k)}\frac{\mu(d/e)}{\phi(kq/e)}\bar{\psi}\left( -\frac{k'k_q}{e}\right)\psi\left( \frac{d}{e}\right)\mu\left( \frac{k'k_q}{e}\right).
	\end{align*}
	Since $\psi$ is a character modulo $q$, we have that only the terms $e=k_qe'$ with $(e',q)=1$ contribute to $\delta(q, kq, d, \psi)$. Thus, only the terms with $d=k_qd'$ such that $(d',q)=1$ contribute to $\delta(q, kq, d, \psi)$, in which case
	\begin{align*}
	\delta(q, kq,d, \psi)= \sum_{e\mid (d',k')} \frac{\mu(d'/e)}{\phi(k'q/e)}\bar{\psi}\left(- \frac{k'}{e}\right)\psi\left(\frac{d'}{e}\right)\mu(k'/e)=\delta(q, k'q, d',\psi).
	\end{align*}
	We see that
\begin{align*}
\mathcal{E}_2&=\sum_{\eta \leq q\leq y} \sideset{}{^*}\sum_{\,\,\psi \pmod q} \tau(\bar{\psi}) \sum_{kq\leq y} \frac{a(kq)}{(kq)^{1-\beta}}\sum_{d\mid kq} \delta(q, kq, d, \psi)\sum_{m\leq Tkq/2\pi d}\frac{b(md)\psi(m)}{(md)^\beta}\\
&=\sum_{\eta \leq q\leq y} \sideset{}{^*}\sum_{\,\,\psi \pmod q} \tau(\bar{\psi}) \sum_{\substack{k_q\leq y\\ k_q\mid q^\infty}}\sum_{\substack{k_qkq\leq y\\ (k, q)=1}} \frac{a(k_qkq)}{(k_qkq)^{1-\beta}}\sum_{\substack{d\mid k}} \delta(q, kq, d, \psi)
\sum_{m\leq Tkq/2\pi d}\frac{b(mdk_q)\psi(m)}{(mdk_q)^\beta}.
\end{align*}
Note that 
\begin{align}
	\delta(q, kq, d, \psi) \ll \sum_{e\mid d} \frac{1}{\phi(kq/e)} \ll \sum_{e\mid d}\frac{e}{kq}(\log \log T)\ll  (\log\log T)^2 dk^{-1}q^{-1}
\end{align}
since $\phi(n)\gg n (\log \log n)^{-1}$ and $\sigma(n)\ll n \log \log n$. Applying this along with the bounds $|a(mn)|\ll |a(m)||a(n)|$ and $|\tau(\psi)|=q^{1/2}$ we have 
\begin{align*}
\mathcal{E}_2
\ll & 
\sum_{\eta \leq q\leq y} \frac{|a(q)|}{q^{3/2}}
\sideset{}{^*}\sum_{\,\,\psi \pmod q} 
\sum_{\substack{k_q\leq y\\ k_q\mid q^\infty}}
\sum_{\substack{k\leq y/qk_q\\ (k, q)=1}}
 \frac{|a(k_qk)|}{k_qk^2}
\Big( \sum_{\substack{d\mid k}} d\Big)
\bigg|\sum_{m\leq Tkq/2\pi d}\frac{b(mdk_q)\psi(m)}{(mdk_q)^\beta}\bigg|
\\
\ll &
\sum_{\eta \leq q\leq y} \frac{|a(q)|}{q^{3/2}}
\sideset{}{^*}\sum_{\,\,\psi \pmod q} 
\sum_{\substack{k_q\leq y\\ k_q\mid q^\infty}}
\sum_{\substack{d\leqs y/qk_q\\ (d, q)=1}}\frac{|a(dk_q)|}{dk_q}
\sum_{\substack{k\leq y/qk_qd\\ (k, q)=1}}
 \frac{|a(k)|}{k^2}
\bigg|\sum_{m\leq Tkq/2\pi }\frac{b(mdk_q)\psi(m)}{(mdk_q)^\beta}\bigg|.
\end{align*}
After grouping $k_q$ and $d$ together and removing the condition $(k,q)=1$, we may upper bound this by 
\[
\sum_{\eta \leq q\leq y} \frac{|a(q)|}{q^{3/2}}
\sideset{}{^*}\sum_{\,\,\psi \pmod q} 
\sum_{d\leqs y/q}\frac{|a(d)|}{d}
\sum_{\substack{k\leq y/qd}}
 \frac{|a(k)|}{k^2}
\bigg|\sum_{m\leq Tkq/2\pi }\frac{b(md)\psi(m)}{(md)^\beta}\bigg|.
\]

	We divide the summation over $k, q,d$ into dyadic intervals 
	$K\leqs k \leqs 2K, Q\leqs q\leqs 2Q, D\leqs d\leqs 2D$ where 
	\begin{equation}
	\label{KQD}
	\eta< Q\leqs y, \qquad KQD\ll y
	\end{equation} 
	to obtain
 \begin{align*}
	\mathcal{E}_2&\ll \sideset{}{'}\sum_{D}\sideset{}{'}\sum_{Q}\sideset{}{'}\sum_{K}
	\sum_{\substack{q\sim Q}} \frac{|a(q)|}{q^{3/2}}
	\sum_{\substack{d\sim D}}\frac{|a(d)|}{d}
	\sum_{\substack{k\sim K} } \frac{|a(k)|}{k^2}
	\sideset{}{^*}\sum_{\,\,\psi \pmod q} \left|\sum_{m\leq kqT/2\pi}\frac{b(md)\psi(m)}{(md)^\beta}\right|.
	\end{align*}
	Here $\sum_{N}'$ is used to indicate the summation of the dyadic partition, so that $\sum_{N}' 1 \ll \log T$, and $\sum_{n\sim N}$ means $\sum_{N\leqs n\leqs 2N}$. 
	Upon bounding by the maximal dyadic sums we find that there exists a $K,Q, D$ satisfying \eqref{KQD} such that 
	 \begin{align*}
	\mathcal{E}_2\ll & (\log T)^3
	\sum_{\substack{d\asymp D}}\frac{|a(d)|}{d}
	\sum_{\substack{k\asymp K} } \frac{|a(k)|}{k^2}
	\sum_{\substack{q\sim Q}} \frac{|a(q)|}{q^{3/2}}
	\sideset{}{^*}\sum_{\,\,\psi \pmod q} \left|\sum_{m\leq kqT/2\pi}\frac{b(md)\psi(m)}{(md)^\beta}\right|.
	\end{align*}
	Applying $a(n)\ll \tau_r(n)(\log n)^C$ and the more brutal bound $a(q)\ll q^{\epsilon}$ we arrive at the following	
 \begin{align}
 \label{E_2 initial bound}
	\mathcal{E}_2\ll
	&  \frac{(\log T)^C}{KQ^{3/2-\epsilon}} 	
	\sum_{\substack{d\asymp D}}\frac{|a(d)|}{d}
	\sum_{\substack{q\sim Q}} 
	\sideset{}{^*}\sum_{\,\,\psi \pmod q} \max_{x\leqs 2KQT}\left|\sum_{m\leqs x }\frac{b(md)\psi(m)}{(md)^\beta}\right|.
	\end{align}
where $b(n)$ is defined in \eqref{Def b}.
%
\subsection{Combinatorial decomposition}	
To evaluate the sum over $m$, we apply Heath-Brown's identity to $\mu$ to decompose $b(n)$ into $O((\log T)^C)$ linear combinations of functions of the form $(f*g)(n)$ where $g$ is supported on integers of short lengths and $g(n)=n^c \psi(n)$ (Type I) or both $f$ and $g$ are supported on integers of short lengths (Type II). For Type I terms, we obtain cancellation from the sum over $\psi$ using P\'olya-Vinogradov inequality. For Type II terms, we obtain cancellation using the large sieve inequality on short Dirichlet polynomials. 

	Let $M(s)=\sum_{n\leq z^{1/J}}\mu(n)n^{-s}$. From Heath-Brown's identity
	\begin{align*}
	\frac{1}{\zeta(s)}=\sum_{1\leq j \leq J} (-1)^{j-1}\binom{J}{j}\zeta(s)^{j-1}M(s)^{j}+ \frac{1}{\zeta(s)}\left(1-M(s)\zeta(s)\right)^J
	\end{align*} 
	we have for $n\leq z$, 
	\begin{align}\label{HeathBrownMu}
	\mu(n)=\sum_{1\leq j\leq J}(-1)^{j-1}\binom{J}{j}1^{(*)(j-1)}*\mu 1_{[1, z^{1/J}]}^{(*)j}.
	\end{align}
	Since $md\ll KQT\ll y T\ll T^{3/2}$, we take $z=T^{3/2}$. 
On splitting each range of summation into dyadic intervals, we see $b(n)$ can be written as a linear combination of $O((\log T)^ {2J+3})$ expressions of the form $f_1*\cdots * f_{2J+3}(n)$, where $f_i$ are supported on dyadic intervals $[F_i/2, F_i]$. For terms in which $f_i$ is absent, we set $F_i=1$, and take $f_i(1)=1, f_i(n)=0, n\geq 2$. For $F_i>1$, we have 
\begin{align*}
&f_i(n)=\mu 1_{[1,z^{1/J}]}(n), \ i=1, \cdots J\\
&f_j(n)=1, \ j=J+1, \cdots, 2J-1, \\
&f_{2J}(n)=n^{-\gamma}, f_{2J+1}(n)=n^{-\alpha}, f_{2J+2}=n^{\beta}, f_{2J+3}=a(n).
\end{align*}
Note that $F_{2J+3}\leq y$ and $F_i \ll T^{3/2J}$ for $i=1, \cdots, J$. 
By Lemma \ref{DecompositionArithmeticFunction}, we write 
\begin{align*}
\frac{b(md)}{(md)^\beta}=\frac{f_1* \cdots * f_{2J+3}(md)}{(md)^\beta}= \sum_{d_1\cdots d_{2J+3}=d}g_1 * \cdots * g_{2J+3}(m)
\end{align*}
with 
\begin{align*}
g_i(m)=g_i(m;d_1, \dots, d_i)
=
\left\{\begin{array}{ll}
f_i(md_i)(md_i)^{-\beta}, & \text{ if } (m, D_i)=1, \text{ where } D_i=d_1 \cdots d_{i-1},\\
0, & \text{ otherwise. }
\end{array}\right.
\end{align*}
Therefore, 
\begin{multline*}
\mathcal{E}_2
 \ll \frac{(\log T)^C}{KQ^{3/2-\epsilon}} \sideset{}{'}\sum_{F_i}\sum_{d\asymp D} \frac{|a(d)|}{d}\sum_{d_1 \cdots d_{2J+3}=d}
\\\times
\sum_{q\sim Q} \sideset{}{^*}\sum_{\,\,\psi \pmod q}
\max_{x\leq 2KQT}\left |\sum_{m\leq x} {(g_1* \cdots * g_{2J+3})(m)\psi(m)} \right|.
\end{multline*}
If $x\ll (yT)^{1/2}$, then we can bound trivially
\begin{align}
&\frac{(\log T)^C}{KQ^{3/2-\epsilon}}\sideset{}{'}\sum_{F_i}\sum_{d\leq y}\frac{|a(d)|}{d}\sum_{d_1 \cdots d_{2J+3}=d} \sum_{q\sim Q} \sideset{}{^*}\sum_{\,\,\psi \pmod q} \left| \sum_{m\leq x} {g_1* \cdots * g_{2J+3}(m)\psi(m)}\right|\nonumber\\
& \ll \frac{Q^2 y^{1/2}T^{1/2+\epsilon}}{KQ^{3/2-\epsilon}}\ll yT^{1/2+\epsilon}.\label{small x}
\end{align}
Thus, we arrive at the bound
\begin{equation}
\label{E_2 bound}
\mathcal{E}_2
 \ll \frac{(\log T)^C}{KQ^{3/2-\epsilon}} \sum_{d\asymp D} \frac{|a(d)|}{d}S(Q,T,d)
\end{equation}
where
 \begin{align}\label{ConvolutionSum}
\mathcal{S}(Q, T, d) =\sideset{}{'}\sum_{F_i}\sum_{d_1 \cdots d_{2J+3}=d}\sum_{q\sim Q}\sideset{}{^*}\sum_{\,\,\psi \pmod q}\max_{(yT)^{1/2}\leq x\leq 2KQT}\left|\sum_{ m\leq x}(g_1*\cdots *g_{2J+4})(m)\psi(m)\right|
 \end{align}
 and where each $g_i$ is supported on $[G_i/2, G_i]$ with $G_i=F_i/d_i$ and $\prod_{i}G_i\ll x$. 
 
 Let $x\gg W\gg x^{2/3}$ be a parameter to be chosen later.
  We see that there exists an $i$ such that $G_i\gg W$ (Type I) or there exists a subset $S\subset\{1,\dots 2J+3\}$ such that $x/W\ll \prod_{i\in S} G_i\ll W$ (Type II). Indeed, if there is an $i$ such that $G_i \gg W$ or $x/W\ll G_i \ll W$ then we are done. Otherwise we may suppose $G_i \ll x/W$ for all $i$. Since $\prod_{i=1}^{2J+3}G_i\gg x\gg W$, there exists an $i_0$ such that $\prod_{i=1}^{i_0}G_i\gg x/W$ and $\prod_{i=1}^{i_0-1}G_i\ll x/W$, and thus  
  \begin{align*}
  \frac{x}{W}\ll \prod_{i=1}^{i_0}G_i\ll \frac{x}{W}\frac{x}{W}\ll W.
  \end{align*}
 \subsection{Type I terms} When $x\gg (yT)^{1/2}\gg y^{3/2}$, we have $ y\ll x^{2/3}\ll W$. By taking $J$ large enough, we also have $z^{1/J} \ll T^{3/2J}\ll T^{1/3}\ll x^{2/3}\ll W$. Thus if there exists an $i$ such that $G_i\gg W$, we must have $i\in \{ J+1, \dots, 2J+2\}$. 
 
By an application of M\"obius inversion, the P\'olya--Vinogradov inequality and partial summation, we see 
\begin{align*}
\sum_{\substack{n_i\sim G_i\\ (n_i, D_i)=1}}n_i^c\psi(n_i)\ll \tau(D_i)q^{1/2}\log q
\end{align*}
uniformly for $c\ll 1/\log T$. 
By grouping the rest of the functions in the $2J+3$ convolution to a function $\tilde{g}$, we see the sum over $m$ in \eqref{ConvolutionSum} becomes
\begin{align*}
&\sum_{mn_i\leq x}\tilde{g}(m)\psi(m) \sum_{\substack{n_i\sim F_i/d_i\\ (n_i, D_i)=1}}(n_id_i)^{c}\psi(n_i)\\
& \ll \frac{x}{W}Q^{1/2}T^\epsilon.
\end{align*}
Therefore, the contribution from Type I terms to $\mathcal{E}_2$ is bounded by 
\begin{align}\label{Type1Error}
\frac{(\log T)^C}{KQ^{3/2-\epsilon}}Q^2 \frac{KQT}{W}Q^{1/2}T^\epsilon= \frac{Q^2T^{1+\epsilon}}{W}.
\end{align}

\subsection{Type II terms} For Type II terms, we use the large sieve inequality to obtain cancellations.  To start with, we have from Perron's formula 
	\begin{align*}
	\sum_{m\leq x}{(g_1* \cdots * g_{2J+3})(m)\psi(m)}=\frac{1}{2\pi i}\int_{\kappa-iU}^{\kappa+iU} B(s, \psi, \vec{d})x^s \frac{ds}{s}+T^\epsilon,
	\end{align*}
	where $\vec{d}=(d_1, \cdots, d_{2J+3}), U=T^{20}$ and $\kappa\asymp 1/\log (KQT)$ and 
	\begin{align*}
	B(s, \psi, \vec{d})=\sum_{m}(g_1*\cdots * g_{2J+3})(m)\psi(m)m^{-s}.
	\end{align*}
	The error from $T^\epsilon$ can be bounded by
	\begin{align}\label{Type2ErrorPerron}
	\frac{(\log T)^C }{KQ^{3/2-\epsilon}}Q^2 T^\epsilon\ll Q^{1/2+\epsilon}T^\epsilon\ll y^{1+\epsilon}
	\end{align}
and thus 
\begin{equation}
\label{SQ bound}
S(Q,T,d)\ll 
\sideset{}{'}\sum_{F_i}\sum_{d_1\ldots d_{2J+3}=d}\sum_{q\sim Q}\sideset{}{^*}\sum_{\,\,\psi \pmod q}\int_{-U}^{U}\frac{\log(KQT)}{1+|t|}|B(\kappa+it,\psi,\vec{d})| dt 
+
y^{1+\epsilon}.
\end{equation}
Let $H_j( \psi ,t)=\sum_{n\sim G_j}g_j(n)\psi(n)n^{-\kappa-it}$ and write
\[B(\kappa+it,\psi,\vec{d})=\prod_{j=1}^{2J+3}H_{j}(\psi,t)=\mathcal{A}(\psi, t)\mathcal{B}(\psi, t)\]
where $\mathcal{A}(\psi, t)=\prod_{i\in S}H_i(\psi, t)$, $\mathcal{B}(\psi, t)=\prod_{i\not \in S}H_i(\psi, t)$ are Dirichlet polynomials of lengths $A, B$ respectively with $A, B\ll W$ from the definition of $S$. 

%
%

It is enough to bound uniformly for $1\leqs V\leqs T^{20}$, 
\begin{align}
\label{AB int}
\frac{1}{V}\sum_{q\sim Q}\sum_{\psi \pmod q}^*\int_{-V}^{V} |\mathcal{A}(\psi, t)\mathcal{B}(\psi, t)|dt.
\end{align}
After an application of the Cauchy-Schwarz inequality and the large sieve inequality
\begin{align*}
\sum_{q\sim Q}\sum_{\psi \pmod q}^* \int_{-V}^V \left|\sum_{m\leq H}h_mm^{-it}\right|^2dt\ll (Q^2V+H)\sum |h_m|^2,
\end{align*}
 we see that \eqref{AB int} is bounded by 
\begin{align*}
&\frac{1}{V}\left(\sum_{q\sim Q} \sum_{\psi \pmod q}^*\int_{-V}^V|\mathcal{A}(\psi, t)|^2dt\right)^{1/2}\left(\sum_{q\sim Q} \sum_{\psi \pmod q}^*\int_{-V}^V|\mathcal{B}(\psi, t)|^2dt\right)^{1/2}\\
&\ll \tau_R(d)(\log T)^CV^{-1}((Q^2V+A)A(Q^2V+B)B)^{1/2}\\
& \ll \tau_R(d)(\log T)^C(AB)^{1/2}(QV^{-1/2}(A+B)^{1/2}+Q^2)+(\log T)^CABV^{-1}
\end{align*}
for some positive integer $R$ since $g_j(m)\ll \tau_{r}(md)(\log md)^{C^\prime}\ll  \tau_{r}(m)\tau_{r}(d)(\log T)^{C^\prime} $.
Applying this in \eqref{SQ bound} and then \eqref{E_2 bound} we see these terms contribute to $\mathcal{E}_2$ at most
\begin{align}
& \frac{(\log T)^{C'}}{KQ^{3/2-\epsilon}}\Big(\sum_{d\asymp D}\frac{|a(d)|\tau_{2J+3}(d)\tau_R(d)}{d}\Big)\nonumber
\\
&\qquad\times\sup_{\substack{1\leqs V\leqs T^{20}\\ AB\ll KQT}}\left((AB)^{1/2}(QV^{-1/2}(A+B)^{1/2}+Q^2)+ABV^{-1}\right)\nonumber\\
&\ll (\log T)^{C'}Q^{\epsilon}\sup_{1\leqs V\leqs T^{20}} \Big(T^{1/2}W^{1/2}V^{-1/2}K^{-1/2}+Q T ^{1/2} K^{-1/2}+ TQ^{-1/2}V^{-1}\Big)\nonumber\\
& \ll (\log T)^{C'}Q^\epsilon\left( T^{1/2}W^{1/2}K^{-1/2}+QT^{1/2}+TQ^{-1/2}\right).\label{Type2Error}
\end{align}


Let $W=(KQT)^{2/3}$. Combining \eqref{small x}, \eqref{Type1Error}, \eqref{Type2ErrorPerron} and \eqref{Type2Error} we have 
\begin{align*}
\mathcal{E}_2\ll y^{4/3} T^{1/3+\epsilon}+ (\log T)^{C'}\left( yT^{1/2+\epsilon}+ T\eta^{-1/2+\epsilon}+y^{1/3}T^{5/6+\epsilon}\right)
\end{align*}
and the result follows.


\end{document}